\documentclass[a4paper,twoside=semi,12pt]{scrartcl}

\usepackage[english]{babel}
\usepackage[utf8]{inputenc}
\usepackage[T1]{fontenc}
\usepackage{ae}
\usepackage{textcomp} %The package supports the Text Companion fonts, which provide many text symbols (such as baht, bullet, copyright, musicalnote, onequarter, section, and yen), in the TS1 encoding.
\usepackage{amsfonts}
\usepackage{amsmath}
\usepackage{amssymb}
\usepackage{csquotes}

\usepackage[amsmath,thmmarks,hyperref,amsthm]{ntheorem} %amsthm.sty conflicts with the definition of theorem layouts in theorem.sty, some features of amsthm.sty have been incorporated into option [amsthm] which has to be used instead of \usepackage{amsthm}.
\usepackage[shortlabels]{enumitem}
\usepackage[pdftex]{hyperref} % Hyperrefs fuer befehl pdflatex
\usepackage{url}

\newcommand{\F}{\mathbb{F}}

\newcommand{\Z}{\mathbb{Z}}

\newcommand{\vek}[1]{\mathbf{#1}}

\newcommand{\qbinom}[3]{\genfrac{[}{]}{0pt}{}{#1}{#2}_{#3}}

\DeclareMathOperator{\GL}{GL}
\DeclareMathOperator{\PGL}{PGL}

\DeclareMathOperator{\PGammaL}{P\Gamma L}

\DeclareMathOperator{\Aut}{Aut}

\DeclareMathOperator{\PG}{PG}

\DeclareMathOperator{\rk}{rk}

\DeclareMathOperator{\drank}{d_{\rk}}
\DeclareMathOperator{\dsubsp}{d}
\DeclareMathOperator{\Inn}{Inn}

\makeatletter
\def\theorem@checkbold{} %schaltet Fettduck für Formeln ab, siehe http://tex.stackexchange.com/questions/61554/formulas-in-bold-theorem-titles-in-ntheorem-package
\renewtheoremstyle{plain}%
{\item[\hskip\labelsep \theorem@headerfont ##1\ ##2 \theorem@separator]}%
{\item[\hskip\labelsep \theorem@headerfont ##1\ ##2\ {\normalfont(##3)}  \theorem@separator]}
\makeatother

\newtheorem{theorem}{Theorem}
\newtheorem{lemma}{Lemma}[section]
\newtheorem{fact}[lemma]{Fact}

\theoremstyle{definition}

\theoremstyle{remark}
\newtheorem{remark}[lemma]{Remark}
\newtheorem{example}[lemma]{Example}

\author{
Thomas Honold%
\thanks{
Department of Information and Electronic Engineering, Zhejiang University, 38 Zheda Road, 310027 Hangzhou, China
\newline
email:~\texttt{honold@zju.edu.cn}
}
\and
Michael Kiermaier%
\thanks{
Mathematisches Institut, Universität Bayreuth, 95440 Bayreuth, Germany
\newline
email:~\texttt{michael.kiermaier@uni-bayreuth.de}
\newline
homepage:~\url{http://www.mathe2.uni-bayreuth.de/michaelk/}
}
\and
Sascha Kurz%
\thanks{
Mathematisches Institut, Universität Bayreuth, 95440 Bayreuth, Germany
\newline
email:~\texttt{sascha.kurz@uni-bayreuth.de}
\newline
homepage:~\url{http://www.wm.uni-bayreuth.de/de/team/Kurz_Sascha/}
}
}

\title{Classification of large partial plane spreads in $\PG(6,2)$\\and related combinatorial objects}

\begin{document}
\maketitle
\begin{abstract}
  The partial plane spreads in $\PG(6,2)$ of maximum
  possible size $17$ and of size $16$ are classified.  Based on this
  result, we obtain the classification of the following closely
  related combinatorial objects: Vector space partitions of $\PG(6,2)$
  of type $(3^{16} 4^1)$, binary $3\times 4$ MRD codes of minimum rank
  distance $3$, and subspace codes with the optimal parameters
  $(7,17,6)_2$ and $(7,34,5)_2$.
\end{abstract}

\subsection*{MSC Classification}
05B25; %Finite geometries
15A21; %Canonical forms, reductions, classification
20B25; %Finite automorphism groups of algebraic, geometric, or combinatorial structures [See also 05Bxx, 12F10, 20G40, 20H30, 51-XX]
%51E23; %Spreads and packing problems
51E14; %Finite partial geometries (general), nets, partial spreads
51E20; %Combinatorial structures in finite projective spaces
94B60 %Other types of codes
%05B05 %Block designs
%94B25 %Combinatorial codes
%51E22 %Linear codes and caps in Galois spaces
%11T71 %Algebraic coding theory; cryptography
%05B40 %Packing and covering

\section{Introduction}\label{sec:intro}

A partial spread in the projective geometry
$\PG(v-1,q)=\PG(\F_q^v/\F_q)$ is a set of mutually disjoint subspaces
of equal dimension. In Finite Geometry, partial spreads arise as
natural generalizations of spreads, i.e., partitions of the point set
of $\PG(v-1,q)$ into subspaces of equal dimension. In the recently
established field of Subspace Coding, which refers to the construction
and analysis of error-correcting codes for noncoherent Random Linear
Network Coding
\cite{Koetter-Kschischang-2008-IEEETIT54[8]:3579-3591,Silva-Kschischang-Koetter-2008-IEEETIT54[9]:3951-3967,Etzion-Vardy-2011-IEEETIT57[2]:1165-1173},
partial spreads arise as constant-dimension codes with the largest
possible minimum distance (equal to twice the dimension of the
codewords).
%\footnote{In \cite{manganiello-gorla-rosenthal08,gorla-ravagnani14} the names ``spread code'' and ``partial spread code'' have been proposed as a clue to the corresponding geometric object.}
In both areas, the
determination of the maximum size of partial spreads with fixed
parameters $q$, $v$ and $k$ (algebraic dimension of the subspaces), and
the corresponding classification of extremal types are of great
importance. 

Specifically, in Subspace Coding partial spreads arise in two
different settings: (1) as ``local'' (or ``derived'') subcodes of
general $(v,M,2\delta;k)$ constant-dimension codes, formed by the
codewords through a fixed subspace of dimension $k-\delta$ (maximum
dimension of the intersection of two distinct codewords); (2) as
constant-dimension layers of unrestricted (``mixed-dimension'')
subspace codes. Our present work, which classifies the partial plane spreads in
$\PG(6,2)$ of maximum size $17$ and of size $16$ and uses this for the
classification of optimal $(7,17,6)_2$ and $(7,34,5)_2$ subspace codes, may
serve as an example for (2). 
The classification of $(7,17,6)_2$ and $(7,34,5)_2$ codes has been
announced without proof in our paper on mixed-dimension
subspace codes \cite{Honold-Kiermaier-Kurz-2016-AiMoC10[3]:649-682}: The
number of isomorphism types of $(7,17,6)_2$ subspace codes is $928$,
and the number of isomorphism types of $(7,34,5)_2$ subspace codes is
$20$. The details of these classifications are provided in this
article.
As an example for (1) we mention the classification of optimal binary $(6,77,4;3)$ subspace codes
\cite{Honold-Kiermaier-Kurz-2016-AiMoC10[3]:649-682}, which used the
classification of maximal partial line spreads in $\PG(4,2)$ as essential
ingredient.
Quite recently, the classification result of this article has been used in a similar way to determine the maximum possible size of $(8,?,6;4)_2$ subspace codes:
In \cite{Heinlein-Kurz-2017-WCC}, $A(8,6;4)_2 \leq 272$ has been shown.
Finally in \cite{Heinlein-Honold-Kiermaier-Kurz-Wassermann-2017-arXiv:1711.06624}, the exact value is determined as $A(8,6;4)_2 = 257$, and moreover the corresponding $(8,257,6;2)_2$ subspace codes are classified into two isomorphism classes.

% hence are involved in virtually every
% existence or classification problem.

% In contrast with the existence problem for spreads,
% which is easily solved\footnote{A spread consisting of $k$-dimensional
%   subspaces of $\F_q^v$ exists if and only if $k\mid v$; see, e.g.,
%   \cite{hirschfeld98}.}, the maximum size of partial spreads---let
% alone the corresponding classification of extremal types---is
% unknown for most parameter sets.\footnote{We should note}

The classification of partial plane spreads in $\PG(6,2)$ of maximum size $17$ and of size $16$ forms the main part of the present work.
While the full classification is inevitably based on computational methods, we point out that for both sizes the feasible hole structures are classified purely by theoretical arguments.
From these results it is comparatively cheap to derive the classification of the above-mentioned subspace codes.
Moreover, the classification of further, related combinatorial objects is carried out, namely of the vector space partitions of $\PG(6,2)$ of type $(3^{16} 4^1)$ and of the (not necessarily linear) binary $3\times 4$ MRD codes of minimum rank distance $3$.

Our work can also be seen as a continuation of~\cite{Gordon-Shaw-Soicher-2004-unpublished} and earlier work of Shaw~\cite{Shaw-2000-DCC21[1-3]:209-222}, where partial line spreads in $\PG(4,2)$ have been classified.
Both cases are instances of $v\equiv 1\pmod{k}$, for which the maximal size of a partial spread has been known since the work of Beutelspacher \cite{Beutelspacher-1975-MathZeit145[3]:211-229}.
The present classification problem, however, is considerably more complex than that in~\cite{Gordon-Shaw-Soicher-2004-unpublished} and can be solved only for sizes close to the maximum size $17$.  Due to the large number of isomorphism classes, only the most important ones can be given in explicit form in this article.
The full data is provided at the online tables of subspace codes~\url{http://subspacecodes.uni-bayreuth.de}.

The remaining part of the paper is structured as follows.
In Section~\ref{sect:pre} the necessary theoretical background is provided.
In particular, partial spreads and the related combinatorial notions of vector space partitions, subspace codes and MRD codes are introduced, together with pointers to the literature.
The classification of partial plane spreads in $\PG(6,2)$ of maximum size $17$ is carried out in Section~\ref{sect:ps17}.
The classification of partial plane spreads of the second largest size $16$ in Section~\ref{sect:ps16} is considerably more involved, but still feasible.
It is based on the classification of the possible hole configurations of such a partial spread in Section \ref{subsect:ps17:hole}.
Section~\ref{sect:mrd} contains the implications of these classifications on MRD codes, and Section~\ref{sect:codes} those on optimal subspace codes.

\section{Preliminaries}
\label{sect:pre}
\subsection{The subspace lattice}
Throughout this article, $V$ is a vector space over $\F_q$ of finite
dimension $v$.  Subspaces of dimension $k$ will be called
\emph{$k$-subspaces} of $V$, or \emph{$(k-1)$-flats} of the projective
geometry $\PG(V)\cong\PG(v-1,q)$.  The set of all $k$-subspaces of $V$
is called the \emph{Gra{\ss}mannian} and denoted by
$\qbinom{V}{k}{q}$. Adopting the projective geometry point-of-view,
the $1$-subspaces of $V$ ($0$-flats) are also called \emph{points},
the $2$-subspaces \emph{lines}, the $3$-subspaces \emph{planes}, the
$4$-subspaces \emph{solids} and the $(v-1)$-subspaces
\emph{hyperplanes}.%
\footnote{Caution: ``$k$-subspace'' refers to the
  algebraic (i.e., vector space) dimension, while ``point'', ``line'',
  ``plane'', etc., are used in the geometric sense.}  As usual,
subspaces of $V$ are identified with the set of points they contain.
The number of all $k$-subspaces of $V$ is given by the Gaussian
binomial coefficient
\[
  \#\qbinom{V}{k}{q}
  = \qbinom{v}{k}{q}
  = \begin{cases}
    \frac{(q^v-1)(q^{v-1}-1)\cdots(q^{v-k+1}-1)}{(q^k-1)(q^{k-1}-1)\cdots(q-1)} & \text{if }k\in\{0,\ldots,v\}\text{;}\\
    0 & \text{otherwise.}
    \end{cases}
\]
The set $\mathcal{L}(V)$ of all subspaces of $V$ together with the
operations $X\cap Y$ (meet) and $X+Y$ (join) forms the
\emph{subspace lattice} of $V$.

After a choice of a basis, we can identify $V$ with $\F_q^v$.  Now for
any $U\in\mathcal{L}(V)$ there is a unique matrix $A$ in reduced row
echelon form such that $U = \langle A\rangle$, where
$\langle .\rangle$ denotes the row space.  Our focus lies on the case
$q = 2$, where the $1$-subspaces
$\langle\mathbf{x}\rangle_{\F_2}\in \qbinom{V}{1}{2}$ are in
one-to-one correspondence with the nonzero vectors
$\mathbf{x} \in V\setminus\{\mathbf{0}\}$.

By the fundamental theorem of projective geometry, for $v\neq 2$ the
automorphism group of $\mathcal{L}(V)$ is given by the natural action
of $\PGammaL(V)$ on $\mathcal{L}(V)$.  The automorphisms contained in
$\PGL(V)$ will be called \emph{linear}.  The action of these groups
provides a notion of (linear) automorphisms and (linear) equivalence
on subsets of $\mathcal{L}(V)$ and in particular on partial spreads,
and vector space partitions, which will be introduced below.

If $q$ is prime, the group $\PGammaL(V)$ reduces to $\PGL(V)$, and for $q = 2$ further to $\GL(V)$.
After a choice of a basis of $V$, its elements are represented by the invertible $v\times v$ matrices $A$, and the action on $\mathcal{L}(V)$ is given by the vector-matrix-multiplication $\mathbf{v} \mapsto \mathbf{v} A$.

Sending a subspace $U$ of $V$ to its orthogonal complement $U^\perp=\{w\in V;b(U,w)=0\}$ with respect to any fixed non-degenerate
bilinear form $b$ on $V$ constitutes an anti-automorphism of $\mathcal{L}(V)$.
Therefore, $\mathcal{L}(V)$ is isomorphic to its dual.

A \emph{projective basis} of $\PG(v-1,q)$ is a set of $v+1$ points of $\PG(v-1,q)$, no $v$ of which lie in a common hyperplane.
In $\PG(\F_q^v)$, a particular projective basis is formed by the points generated by the unit vectors $\vek{e}_1,\ldots,\vek{e}_v$ together with the all-one vector $\vek{e}_1 + \ldots + \vek{e}_v$.
It can be shown that the group $\PGL(v,q)$ acts regularly on the set of all ordered projective bases of $\PG(\F_q^v)$.
Therefore, $\PG(v-1,q)$ has a projective basis, which is unique up to collineations.

\subsection{Partial spreads}
A \emph{$k$-spread} of $V$, or of the corresponding projective
geometry $\PG(V)$, is a partition of the point set
$\qbinom{V}{1}{q}$ into $k$-subspaces.%
\footnote{Thus, e.g., ``$3$-spread'' and ``plane spread'' (spread consisting of planes)
  are synonymous.}
As is well-known, a $k$-spread in $V$ exists if and only
if $k\mid v$; cf.\ 
%\cite{Andre-1956-MathZeit60[1]:156-186,Hirschfeld-1998}.
\cite[\S VI]{Segre-1964-AnnDMPA64[1]:1-76} or \cite[Ch.~4.1]{Hirschfeld-1998}.
Weakening the above definition, a set $\mathcal{S}$ of $k$-subspaces of $V$ is called a \emph{partial $k$-spread} if the intersection of any two elements (\emph{blocks}) of $\mathcal{S}$ is the trivial subspace $\{0\}$.
%The elements of $\mathcal{S}$ are called \emph{blocks}.  
The primary problem in the theory of partial spreads is to determine
the maximum possible size of a partial $k$-spread in $V$ for all
parameter tuples $q,v,k$. Following standard terminology in Finite
Geometry, partial spreads attaining the maximum size are called
\emph{maximal partial spreads}, while inclusion-maximal partial
spreads are called \emph{complete partial
  spreads}.%
\footnote{Thus a partial $k$-spread $\mathcal{S}$ is complete
if and only if it is not \emph{extendible} to a partial spread
$\mathcal{S'}$ properly containing $\mathcal{S}$; equivalently, no
$k$-subspace of $V$ is disjoint from all members of $\mathcal{S}$.}

Contributions to the problem of determining the maximum sizes of
partial spreads have been made in
\cite{Hong-Patel-1972-IEEEToC_C21[12]:1322-1331,Beutelspacher-1975-MathZeit145[3]:211-229,ElZanati-Jordon-Seelinger-Sissokho-Spence-2010-DCC54[2]:101-107}
and more recently
\cite{Kurz-2017-DCC85[1]:97-106,Nastase-Sissokho-2017-DM340[7]:1481-1487,Kurz-2017-AustAsJComb68[1]:122-130,Honold-Kiermaier-Kurz-2018-COST}.
For surveys on this topic, see \cite{Eisfeld-Storme-2000,Honold-Kiermaier-Kurz-2018-COST}. 

The related problem of classifying maximal partial spreads into
isomorphism types with respect to the action of the collineation group
of $\PG(V)$ (isomorphic to the projective semilinear group
$\PGammaL(v,q)$) is generally much more difficult. For the few
known classifications in the spread case, we refer to the references
mentioned in
%\cite[Theorem 3.1]{Honold-Kiermaier-Kurz-2015-ContempM632:157-176}. 
\cite[Th.~3.1]{Honold-Kiermaier-Kurz-2016-AiMoC10[3]:649-682}.
The myriads of line spreads in $\PG(5,2)$ have been classified in \cite{Mateva-Topalova-2009-JCD17[1]:90-102}.

Points in $\qbinom{V}{1}{q}$ not covered by the blocks of a partial spread $\mathcal{S}$ are called \emph{holes} of $\mathcal{S}$.
Considering $\mathcal{S}$ as fixed, we will denote the set of holes of $\mathcal{S}$ by $N$ and refer to the $\F_q$-span $\langle N\rangle$ as the \emph{hole space} (\emph{Hohlraum} in German) of $\mathcal{S}$.

%Its dimension will be denoted by $n$.

As it has turned out, a reasonable approach to the investigation of
partial plane spreads is provided by the analysis of the possible
dimensions of the hole space and the subsequent classification of the
possible hole configurations. For any point set
$X \subseteq \qbinom{V}{1}{q}$, the number of holes in $X$ is denoted
by $h(X)$ and called the \emph{(hole) multiplicity} of $X$.  Obviously,
$h(X) \leq h(Y)$ for all $X \subseteq Y \subseteq \qbinom{V}{1}{q}$,
and $h(X) = h(N) = h(V)$ for any subset $X$ containing $N$.  The map
$X\mapsto h(X)$ coincides with the extension of the (multi-)set
$P\mapsto h(P)$ to the power set of $\qbinom{V}{1}{q}$ in the sense of
\cite[Def.~11]{Dodunekov-Simonis-1998-ElecJComb5:R37} and
\cite[Eq.~(18)]{Honold-Landjev-2000-ElecJComb7:R11}.

In the following, the number of hyperplanes in $\PG(V)$ containing
precisely $i$ blocks is denoted by $a_i$.  The sequence
$(a_i)_{i\geq 0}$ is called the \emph{spectrum} of $\mathcal{S}$.
Often, it is convenient to give spectra in exponential notation
$(1^{a_1} 2^{a_2} \ldots)$, where expressions $i^{a_i}$ with $a_i = 0$
may be skipped.

\subsection{Subspace codes}
\label{subsect:subspace_codes}
The \emph{subspace distance} on $\mathcal{L}(V)$ is defined as
\[
	\dsubsp(U_1,U_2)
	= \dim(U_1) + \dim(U_2) - 2\dim(U_1 \cap U_2)
	= 2\dim(U_1 + U_2) - \dim(U_1) - \dim(U_2)
\]
The subspace distance is just the graph-theoretic distance of $U_1$ and $U_2$ in the subspace lattice $\mathcal{L}(V)$.
Any set $\mathcal{C}$ of subspaces of $V$ is called a \emph{subspace code}.
The dimension $v$ of its ambient space $V$ is called the \emph{length} of $\mathcal{C}$, and the elements of $\mathcal{C}$ are called \emph{codewords}.
Its \emph{minimum distance} is $\dsubsp(\mathcal{C}) =
\min(\{\dsubsp(B_1, B_2) \mid B_1,B_2\in\mathcal{C},B_1\neq B_2\})$.
We denote the parameters of $\mathcal{C}$ by $(v,\#\mathcal{C},\dsubsp(\mathcal{C}))_q$.
The maximum size of a $(v,?,d)_q$ subspace code is denoted by $A_q(v,d)$.
See \cite{Etzion-Storme-2016-DCC78[1]:311-350} for an overview on what
is known about these numbers and 
\cite{Honold-Kiermaier-Kurz-2016-AiMoC10[3]:649-682}
for several recent results.

An important class of subspace codes are the \emph{constant dimension
  (subspace) codes}, where all codewords are subspaces of the same
dimension $k$.  In this case, we add the parameter $k$ to the
parameters and say that $\mathcal{C}$ is a
$(v,\#\mathcal{C},\dsubsp(\mathcal{C});k)_q$ constant dimension code.
The maximum size of a $(v,?,d;k)_q$ constant dimension code is denoted
by $A_q(v,d;k)$.
The subspace distance of two $k$-subspaces
$U_1, U_2$ can be stated as
$\dsubsp(U_1, U_2) = 2(k - \dim(U_1 \cap U_2)) = 2(\dim(U_1 + U_2) -
k)$.  In particular, the minimum distance of a constant dimension code
is of the form $d=2\delta\in2\Z$.

Setting $t = k-\delta + 1$, we get the alternative characterization of a $(v,?,2\delta;k)_q$ constant dimension code $\mathcal{C}$ as a set of $k$-subspaces of $V$ such that each $t$-subspace is contained in at most one codeword of $\mathcal{C}$.
This gives a connection to the notion of a $t$-$(v,k,\lambda)_q$ subspace design ($q$-analog of a combinatorial design), which is defined as a subset $\mathcal{D}$ of $\qbinom{V}{k}{q}$ such that each $t$-subspace is contained in exactly $\lambda$ elements of $\mathcal{D}$.
For example, in this way a $k$-spread in $\PG(v-1,q)$ is the same as a $1$-$(v,k,1)_q$ design.
For our analysis of the hole set of partial spreads, we will start with the spectrum and look at the set of blocks contained in a fixed hyperplane.
For (subspace) designs, such a set of blocks is called the \emph{residual design} \cite{Kiermaier-Laue-2015-AiMoC9[1]:105-115}.
A survey on subspace designs can be found in \cite{Braun-Kiermaier-Wassermann-2018-COST1} and, with a focus on computational methods, in \cite{Braun-Kiermaier-Wassermann-2018-COST2}.

The isometry group of $(\mathcal{L}(V),\dsubsp)$ is given by all
automorphisms and antiautomorphisms of $\mathcal{L}(V)$.  For
$v\geq 3$ it is of type $\PGammaL(v,q) \rtimes \Z/2\Z$, where the
$\Z/2\Z$-part is generated by an anti-automorphism of
$\mathcal{L}(V)$.  The action of this group provides the natural
notion of isomorphisms and automorphisms of subspace codes.  If the
action is restricted to $\PGammaL(v,q)$, we will use the terms
\emph{inner} isomorphisms and \emph{inner} automorphisms.
Inner automorphisms are precisely the lattice automorphisms of $\mathcal{L}(V)$,  
while the remaining isomorphisms reverse the order of the lattice.
A subspace code which is isomorphic to its dual by an inner isomorphism
will be called \emph{iso-dual}.

For the most recent numeric lower and upper bounds on $A_q(v,d)$ and $A_q(v,2\delta;k)$, we refer to the online tables of subspace codes at \url{http://subspacecodes.uni-bayreuth.de}.
See \cite{Heinlein-Kiermaier-Kurz-Wassermann-2016-arXiv:1601.02864} for a brief manual and description of the implemented methods.

\subsection{MRD codes}
Let $m$, $n$ be positive integers.
The \emph{rank distance} of $m\times n$ matrices $A$ and $B$ over $\F_q$ is defined as $\drank(A,B) = \rk(A - B)$.
The rank distance provides a metric on $\F_q^{m\times n}$.
Any subset $\mathcal{C}$ of the metric space $(\F_q^{m\times n},\drank)$ is called \emph{rank metric code}.
Its \emph{minimum distance} is $\drank(\mathcal{C}) = \min(\{\drank(A,B) \mid \{A,B\}\in\binom{\mathcal{C}}{2}\})$.
If $\mathcal{C}$ is a subspace of the $\F_q$-vector space $\F_q^{m\times n}$, $\mathcal{C}$ is called \emph{linear}.

If $m\leq n$ (otherwise transpose), $\#\mathcal{C} \leq q^{(m-d+1)n}$ by \cite[Th.~5.4]{Delsarte-1978-JCTSA25[3]:226-241}.
Codes achieving this bound are called \emph{maximum rank distance (MRD) codes}.
In fact, MRD codes do always exist.
A suitable construction has independently been found in \cite{Delsarte-1978-JCTSA25[3]:226-241,Gabidulin-1985-PInfTr21[1]:1-12,Roth-1991-IEEETIT37[2]:328-336}.
Today, these codes are known as \emph{Gabidulin} codes.
In the square case $m = n$, after the choice of a $\F_q$-basis of $\F_{q^n}$ the Gabidulin code is given by the matrices representing the $\F_q$-linear maps given by the $q$-polynomials $a_0 x^{q^0} + a_1 x^{q^1} + \ldots + a_{n-d} x^{q^{n-d}}\in\F_{q^n}[x]$.
Recently, some progress on the study of MRD codes has been made:
The algebraic structure of MRD codes has been analyzed in \cite{DeLaCruz-Kiermaier-Wassermann-Willems-2016-AiMoC10[3]:499-510}.
New examples of MRD codes have been constructed in~\cite{Kshevetskiy-Gabidulin-2005-PIISIT2000:446,HorlemannTrautmann-Marshall-2017-AiMoC11[3]:533-548,Sheekey-2016-AiMoC10[3]:475-488,Liebhold-Nebe-2016-ArchMath107[4]:355-366}.

The automorphisms of the metric space $(\F_q^{m\times n},\drank)$ are given by the mappings $A \mapsto P\sigma(A)Q + R$ with $P\in\GL(m,q)$, $Q\in\GL(n,q)$, $R\in\F_q^{m\times n}$ and $\sigma\in\Aut(\F_q)$, and in the square case $m=n$ additionally $A\mapsto P\sigma(A^\top)Q + R$, see \cite{Hua-1951-ActaMSinica1:109-163}, \cite{Wan-1962-SciSinica11:1183-1194} and \cite[Th.3.4]{Wan-1996-GeometryOfMatrices}.
The automorphisms of the first type will be called \emph{inner} and denoted by $\Inn(m,n,q)$.
The action of these groups provides a notion of (inner) automorphisms and equivalence of rank metric codes.
In the non-square case $m \neq n$, any automorphism is inner.

An (inner) isomorphism class $X$ of rank metric codes will be called \emph{linear} if it contains a linear representative.
Otherwise, $X$ is called \emph{non-linear}.
In a linear isomorphism class, the linear representatives are exactly those containing the zero matrix.
Hence, we can check $X$ for linearity by picking some representative $\mathcal{C}$ and some $B\in\mathcal{C}$ and then testing the translated representative $\{A - B \mid A\in\mathcal{C}\}$ of $X$ (which contains the zero matrix) for linearity.

The \emph{lifting} map $\Lambda : \F_q^{m\times n} \to \mathcal{L}(\F_q^{m+n})$ maps an $(m\times n)$-matrix $A$ to the row space $\langle(I_m \mid A)\rangle$, where $I_m$ denotes the $m\times m$ identity matrix.
In fact, the lifting map is an isometry $(\F_q^{m\times n},2\drank) \to (\mathcal{L}(\F_q^{m+n}),\dsubsp)$.
Thus, for any $m\times n$ rank metric code $\mathcal{C}$ of size $M$ and minimum distance $\delta$, the \emph{lifted code} $\Lambda(\mathcal{C})$ is a $(m+n,M,2\delta;m)_q$ constant dimension code.
Of particular interest are the lifted MRD codes, which are constant dimension codes of fairly large, though not maximum size.

\begin{fact}
\label{fct:mrd_structure}
\begin{enumerate}[(a)]
\item
Let $\mathcal{C}$ be an $m\times n$ MRD code of minimum distance $\delta$.
Then $\Lambda(\mathcal{C})$ is an $(m+n,q^{(m-\delta+1)n},2\delta;m)_q$ constant dimension code.
Denoting the span of the unit vectors $\vek{e}_{m+1},\ldots,\vek{e}_{m+n}$ in $\F_q^{m+n}$ by $S$, we have $\dim(S) = n$ and each codeword of $\Lambda(\mathcal{C})$ has trivial intersection with $S$.
Moreover, setting $t = m-\delta+1$, each $t$-subspace of $V$ having trivial intersection with $S$ is contained in a unique codeword of $\mathcal{C}$.
\item
Let $\mathcal{D}$ be an $(m+n,q^{(m-\delta+1)n},2\delta;m)_q$ subspace code such that there exists an $n$-subspace $S$ having trivial intersection with all codewords of $\mathcal{D}$.
Then $\mathcal{D}$ is equivalent%
\footnote{The requirement for the new basis $\{\vek{b}_1,\ldots,\vek{b}_{m+n}\}$ is that $S = \langle\vek{b}_{m+1},\ldots,\vek{b}_{m+n}\rangle$.}
to $\Lambda(\mathcal{C})$ with an $m\times n$ MRD code $\mathcal{C}$ of minimum distance $\delta$.
\end{enumerate}
\end{fact}

By the above fact, a lifted MRD code in some sense optimally packs $V\setminus S$.
However, by considering additional codewords intersecting $S$ nontrivially, the lifted code usually can be extended without destroying the minimum distance.
In this article, we focus on the case of binary $3\times 4$ MRD codes of minimum rank distance $3$, where only a single codeword can be added.%
\footnote{For the additional codeword, any plane contained in the solid $S$ of Fact~\ref{fct:mrd_structure} can be taken.}
The question for the largest possible extension of a general lifted MRD code is an open problem.
Upper bounds have been published in \cite[Th.~10, Th.~11]{Etzion-Silberstein-2013-IEEETIT59[2]:1004-1017} and generalized in \cite{Heinlein-2018-arXiv:1801.04803}.
Lower bounds and constructions are found in \cite{Etzion-Silberstein-2013-IEEETIT59[2]:1004-1017,Heinlein-Kurz-2017-IEEETIT63[12]:7651-7660, Heinlein-2018-arXiv:1801.04803}.

%In the case $d = m$, this implies that $\mathcal{C} \cup \{S\}$ is a vector space partition of type $(m^{q^{(m-d+1)n}} n^1)$

\subsection{Vector space partitions}
A \emph{vector space partition} of $V$ is a set $\mathcal{P}$ of subspaces of $V$ partitioning the point set of $\PG(V)$.
The \emph{type} of $\mathcal{P}$ is $(1^{n_1} 2^{n_2} \ldots)$, where $n_i$ denotes the number of elements of $\mathcal{P}$ of dimension $i$.
An important problem is the characterization of the realizable types of vector space partitions, see e.g. \cite{Heden-2009-DM309[21]:6169-6180,ElZanati-Heden-Seelinger-Sissokho-Spence-VandenEynden-2010-JCD18[6]:462-474,Heden-2012-DiscreteMathAlgorithmsAppl4[1]:1250001,Seelinger-Sissokho-Spence-VandenEynden-2012-FFA18[6]:1114-1132}.

A partial $k$-spread is the same as a vector space partition of type $(1^{n_1} k^{n_k})$.
In Lemma~\ref{lem:mrd_to_vsp}, we will show that $m\times n$ MRD codes with $m\leq n$ of minimum rank distance $m$ over $\F_q$ are essentially the same as vector space partitions of $\F_q^{m+n}$ of the type $(m^{q^n} n^1)$.

\subsection{Partial plane spreads in $\PG(6,2)$}
\label{sec:ps}
From now on, we investigate partial plane spreads $\mathcal{S}$ in
$\PG(6,2)$, so we specialize to $q=2$, $v=7$ and $k=3$.
It is known that the maximum partial plane spreads are of size $A_2(7,6;3) = A_2(7,6) = 17$ \cite{Hong-Patel-1972-IEEEToC_C21[12]:1322-1331,Beutelspacher-1975-MathZeit145[3]:211-229,Honold-Kiermaier-Kurz-2016-AiMoC10[3]:649-682}.

In the following, we prepare some arguments needed later for the classification of the hole structure in the cases $\#\mathcal{S}\in\{16,17\}$.
Of course, the reasoning can easily be translated to general partial spreads.

\begin{lemma}
	\label{lem:ps_hyperplane_holes}
	Let $H$ be a hyperplane in $\PG(V)$ containing $i$ blocks of $\mathcal{S}$.
	Then $i \leq \left\lfloor\frac{63 - 3\#\mathcal{S}}{4}\right\rfloor$ and $h(H) = 63 - 3\#\mathcal{S} - 4i$.
\end{lemma}

\begin{proof}
	The $i$ blocks cover $\qbinom{3}{1}{2}\cdot i = 7i$ points of $H$.
	The intersection of any of the remaining $\#\mathcal{S}-i$ blocks with $H$ is a line, so their intersections cover $\qbinom{2}{1}{2}\cdot (\#\mathcal{S}-i) = 3(\#\mathcal{S} - i)$ further points of $H$.
	The remaining points of $H$ must be holes, so
	\[
		h(H) = \qbinom{6}{1}{2} - 7i - 3(\#\mathcal{S}-i) = 63 - 3\#\mathcal{S} -4i\text{.}
	\]
	The upper bound on $i$ follows since $h(H)$ cannot be negative.
\end{proof}

\begin{lemma}
	The spectrum $(a_i)_{i\geq 0}$ of $\mathcal{S}$ satisfies the \emph{standard equations}
	\label{lem:ps_spec}
	\begin{align*}
		\sum_i a_i & = 127\text{,} \\
		\sum_i i a_i & = 15\#\mathcal{S}\text{,} \\
		\sum_i \binom{i}{2} a_i & = \binom{\#\mathcal{S}}{2}\text{.}
	\end{align*}
\end{lemma}

\begin{proof}
	The first equation is simply the observation that each of the $\qbinom{7}{6}{2} = 127$ hyperplanes is counted exactly once by the $a_i$.
	The second equation arises from double counting the pairs $(H,K)\in\qbinom{V}{6}{2}\times\mathcal{S}$ with $K\leq H$ and the fact that each block is contained in exactly $\qbinom{7-3}{6-3}{2} = 15$ hyperplanes.
	The third equation arises from double counting the pairs $(H,\{K_1,K_2\}) \in\qbinom{V}{6}{2}\times \binom{\mathcal{S}}{2}$ with $K_1 \leq H$ and $K_2\leq H$ and the fact that any pair of distinct blocks of $\mathcal{S}$ is disjoint and therefore contained in a unique hyperplane in $\PG(V)$.
\end{proof}

Furthermore, we will make use of the \emph{hole spectrum}
$(b_i)_{i\geq 0}$, where $b_i$ is the number of hyperplanes in
$\PG(\langle N\rangle)$ containing $i$ holes.  The following lemma
shows that the hole spectrum is determined by the spectrum and
$\dim\langle N\rangle$.

\begin{lemma}
	\label{lem:spec2holespec}
	For all $j \geq 0$,
	\[
		b_j
		= \begin{cases}
			\frac{1}{2^{7 - \dim\langle N\rangle}} \cdot a_{(63 - 3\#\mathcal{S} - j)/4} & \text{if }j < h(\langle N\rangle)\text{ and }j \equiv \#\mathcal{S} - 1\bmod 4\text{;}\\
			0 & \text{otherwise.}
		\end{cases}
	\]
\end{lemma}

\begin{proof}
	We have $b_{h(N)} = 0$, since otherwise there exists a hyperplane $T$ of $\langle N\rangle$ containing all the holes contradicting the definition of the span.

	For any hyperplane $H$ in $\PG(V)$ not containing $\langle N\rangle$, the subspace $T = H \cap \langle N\rangle$ is the only hyperplane of $\langle N\rangle$ contained in $H$.
	We have that $h(H) = h(T)$.
	
	On the other hand given a hyperplane $T$ in $\PG(\langle
        N\rangle)$, the number of hyperplanes in $\PG(V)$ containing
        $T$ but not $\langle N\rangle$, is 
	\begin{multline*}
		\qbinom{\dim(V) - \dim(T)}{(\dim(V)-1) - \dim(T)}{2} - 
		\qbinom{\dim(V) - \dim\langle N\rangle}{(\dim(V)-1) - \dim\langle N\rangle}{2} \\
		= \qbinom{8-\dim\langle N\rangle}{1}{2} - \qbinom{7-\dim\langle N\rangle}{1}{2} = 2^{7-\dim\langle N\rangle}\text{.}
	\end{multline*}
	Therefore for all $j < h(N)$,
	\[
		\#\left\{H \in \qbinom{V}{6}{2} \mid h(H) = j\right\}
		= 2^{7-\dim\langle N\rangle} \cdot \#\left\{T \in \qbinom{\langle N\rangle}{\dim\langle N\rangle-1}{2} \mid h(T) = j\right\}\text{.}
	\]
	The application of Lemma~\ref{lem:ps_hyperplane_holes} concludes the proof.
\end{proof}

\subsection{Scientific software}
In the computational parts of our work, the following software packages have been used:
\begin{itemize}
\item Computations in vector spaces and matrix groups: magma~\cite{Magma}.
\item Enumeration of maximum cliques: cliquer~\cite{Niskanen-Ostergard-2003-cliquer}.
\item Enumeration of solutions of exact cover problems: libexact~\cite{Kaski-Pottonen-2008-libexact}, based on the \enquote{dancing links} algorithm~\cite{Hitotumatu-Noshita-1979-IPL8[4]:174-175,Knuth-2000-dancing_links}.
\item Computation of canonical forms and automorphism groups of sets of subspaces: The algorithm in~\cite{Feulner-2013-arXiv:1305.1193} (based on~\cite{Feulner-2009-AiMoC3[4]:363-383}, see also~\cite{Feulner-2013-Thesis}).
\end{itemize}

\section{Maximum partial plane spreads in $\PG(6,2)$}
\label{sect:ps17}
In this section, $\mathcal{S}$ is a maximum partial plane spread in
$\PG(V) \cong \PG(6,2)$, i.e. a partial plane spread of size $17$.
The number of holes is $h(V) = 127 - 17\cdot 7 = 8$.  We are going to
prove the following

\begin{theorem}
	\label{thm:size17}
	There are $715$ isomorphism types of maximum partial plane spreads $\mathcal{S}$ in $\PG(6,2)$.
	In all cases, the hole set $N$ is an affine solid.

	The intersections of the blocks of $\mathcal{S}$ with the
        plane $E = \langle N\rangle \setminus N$ yield a vector space
        partition of $E$ whose dimension distribution will be called
        the \emph{type} of $\mathcal{S}$.  The $715$ isomorphism types
        fall into $150$ of type $(3^1)$, $180$ of type $(2^1 1^4)$ and
        $385$ of type $(1^7)$.
\end{theorem}

We remark that the partial plane spreads of type $(3^1)$ correspond to the lifted $3\times 4$ MRD codes of minimum distance $3$, extended by a single codeword.

\begin{lemma}
	\label{lem:ps17_spec}
	The spectrum of $\mathcal{S}$ is given by $(1^7 2^{112} 3^8)$.
\end{lemma}

\begin{proof}
	By Lemma~\ref{lem:ps_hyperplane_holes}, for a hyperplane $H$ containing $i$ blocks of $\mathcal{S}$ we have $i \leq 3$ and $h(H) = 12 - 4i$.
	Since $h(H)$ cannot exceed the total number $h(V) = 8$ of holes, additionally we get $i \geq 1$.
	Now Lemma~\ref{lem:ps_spec} yields the linear system of equations
	\[
		\begin{pmatrix}
			1 & 1 & 1 \\
			1 & 2 & 3 \\
			0 & 1 & 3
		\end{pmatrix}
		\begin{pmatrix}
			a_1 \\ a_2 \\ a_3
		\end{pmatrix}
		=
		\begin{pmatrix}
			127 \\
			255 \\
			136
		\end{pmatrix}
	\]
	with the unique solution
	\[
		(a_1,a_2,a_3) = (7,112,8)\text{.}
	\]
\end{proof}

\begin{lemma}
	\label{lem:ps17_holes}
	The $8$ holes form an affine solid in $V$.
\end{lemma}

\begin{proof}
	By $h(H) = 12-4i$ and Lemma~\ref{lem:ps17_spec}, there are $a_1 = 7$ hyperplanes of $V$ containing all $8$ holes of $\mathcal{S}$.
	Since this number equals the number of hyperplanes containing $N$, we get
	\[
		\qbinom{7-\dim\langle N\rangle}{6 - \dim\langle N\rangle}{2} = 7
	\]
	with the unique solution $\dim\langle N\rangle = 4$.
	By Lemma~\ref{lem:spec2holespec}, the spectrum given in Lemma~\ref{lem:ps17_spec} translates to the hole spectrum $(0^1 4^{14})$.
	Thus, there is a single plane $E$ in $\langle N\rangle$ without any holes.
	So the affine solid $\langle N\rangle\setminus E$ consists of the $8$ holes.
\end{proof}
\begin{remark}
  \label{rmk:ps17_holes}
  Lemma~\ref{lem:ps17_holes} generalizes to maximal partial plane
  spreads in $\PG(v-1,2)$ with $v\equiv 1\pmod{3}$, $v\geq 7$. It is
  known that any such maximal partial plane spread has size
  $(2^v-9)/7$ and hence $8$
  holes\cite{Beutelspacher-1975-MathZeit145[3]:211-229}, and these
  form an affine solid as well. This follows from a more general (but
  entirely straightforward) analysis along the lines of
  Section~\ref{sec:ps} and the preceding two lemmas.
\end{remark}

To reduce the search space to a feasible size, good substructures are needed as starting configurations.
For the classification of $(6,77,4;3)_2$ constant dimension codes in~\cite{Honold-Kiermaier-Kurz-2015-ContempM632:157-176}, \enquote{$17$-configurations} have proven to provide suitable starting configurations.
Modifying this approach for the present situation, we call a set $\mathcal{T}$ of $5$ pairwise disjoint planes in $\PG(V)$ a \emph{$5$-configuration} if there are two hyperplanes $H_1\neq H_2$ both containing $3$ elements of $\mathcal{T}$.
Since $\dim(H_1 \cap H_2) = 5$, in this situation $H_1 \cap H_2$ contains exactly one element of $\mathcal{T}$.

\begin{lemma}
	The partial spread $\mathcal{S}$ contains a $5$-configuration.
\end{lemma}

\begin{proof}
	By the spectrum, there are $8$ hyperplanes containing three blocks.
	From $3\cdot 8 = 24 > \#\mathcal{S}$, the $8$ sets of blocks covered by these hyperplanes cannot be pairwise disjoint.
\end{proof}

The above lemma allows us to classify the maximum partial plane spreads in $\PG(6,2)$ by first generating all $5$-configurations up to isomorphisms and then to enumerate all extensions of a $5$-configuration to a maximum partial spread.
For the $5$-configurations, we fixed a block $B$ and two hyperplanes $H_1 \neq H_2$ passing through $B$, which is unique up to isomorphisms.
Then, we enumerated all extensions to a $5$-configuration by adding two blocks in $H_1$ and two blocks in $H_2$ up to isomorphism.
We ended up with six types of $5$-configurations.

Formulating the extension problem as a maximum clique problem, we computed the number of possible extensions as $2449$, $2648$, $3516$, $3544$, $3762$ and $25840$.
Filtering out isomorphic copies, we end up with $715$ isomorphism types of partial plane spreads in $\PG(6,2)$.
Their type was determined computationally.

\section{Partial plane spreads in $\PG(6,2)$ of size $16$}
\label{sect:ps16}
Now $\#\mathcal{S} = 16$.
We are going to prove the following result.

\begin{theorem}
	\label{thm:size16}
	There are $14445$ isomorphism types of partial plane spreads in $\PG(6,2)$ of size $16$.
	Among them, $3988$ are complete and $10457$ are extendible to size $17$.
	\begin{enumerate}[(a)]
		\item\label{thm:size16:ext}
		The hole set $N$ of the extendible partial plane spreads is the disjoint union of a plane $E$ and an affine solid $A$.
		\begin{enumerate}[(i)]
			\item\label{thm:size16:ext:n4}
				In $37$ cases, $\dim(\langle N\rangle) = 4$ and $\dim(E \cap \langle A\rangle) = 3$.
			In other words, $N$ is a solid.
			\item\label{thm:size16:ext:n5}
			    In $69$ cases, $\dim(\langle N\rangle) = 5$ and $\dim(E \cap \langle A\rangle) = 2$.
			In other words, $N$ is the union of three planes $E_1,E_2,E_3$ passing through a common line $L$ such that $E_1/L, E_2/L, E_3/L$ are in general position in the factor geometry $\PG(V/L)\cong\PG(4,2)$.
			\item\label{thm:size16:ext:n6}
				In $3293$ cases, $\dim(\langle N\rangle) = 6$ and $\dim(E \cap \langle A\rangle) = 1$.
			\item\label{thm:size16:ext:n7}
				In $7058$ cases, $\dim(\langle N\rangle) = 7$ and $\dim(E \cap \langle A\rangle) = 0$.
		\end{enumerate}
		\item 
			\label{thm:size16:complete}
	The hole set $N$ of the complete partial plane spreads is the union of $7$ lines $L_1,\ldots,L_7$ passing through a common point $P$, such that $\{L_1/P, \ldots, L_7/P\}$ is a projective basis of the factor geometry $\PG(V/P) \cong \PG(5,2)$. In particular, $\dim(\langle N\rangle) = 7$.
	\end{enumerate}
\end{theorem}

\begin{remark}
	In the cases where spreads exist, that is $k\mid v$, it is known that partial spreads of a size close to the size of a spread always can be extended to a spread, see for example \cite[Sect.~4]{Eisfeld-Storme-2000}.
	In the case of partial line spreads in $\PG(4,2)$, the same is true: All $9$ types of partial spreads of size $8$ are extendible to a partial spread of maximum possible size $9$ \cite{Gordon-Shaw-Soicher-2004-unpublished}.
	However, the existence of complete partial plane spreads of size $16$ in Theorem~\ref{thm:size16} suggests that for $k\nmid v$ the answer may be entirely different from that for $k\mid v$.
\end{remark}

\begin{remark}
  The $5$ possibilities for the structure of the hole set in
  Theorem~\ref{thm:size16} are unique up to isomorphism. Again the
  structure remains the same for partial plane spreads of size
  $(2^v-16)/7$ (one less than the maximum size) in $\PG(v-1,2)$,
  $v\equiv 1\pmod{3}$, $v\geq 7$. The key step in the proof of this
  result is the observation that $n=\dim(\langle N\rangle)\leq 7$ for
  any such $v$, since the $15$ holes determine a doubly-even binary
  linear $[15,n]$ code. For details about the links
  between partial spreads and divisible linear codes we refer to \cite{Honold-Kiermaier-Kurz-2018-COST}.
  % we refer to the forthcoming article \cite{dsmta:q-r-divisble}.
\end{remark}

For $\#\mathcal{S} = 16$, the number of holes is $\#N = 15$.
By Lemma~\ref{lem:ps_hyperplane_holes}, any hyperplane $H$ in $\PG(6,2)$ contains $i\in\{0,1,2,3\}$ blocks and $15-4i$ holes of $\mathcal{S}$.
So $a_0$ is the number of hyperplanes containing all the holes of $\mathcal{S}$.
Since this equals the number of hyperplanes containing $\langle N\rangle$, $a_0$ is of the form $\qbinom{7-\dim\langle N\rangle}{6-\dim\langle N\rangle}{2} = \qbinom{7-\dim\langle N\rangle}{1}{2}$.
Since $\langle N\rangle$ contains the $15 = \qbinom{4}{1}{2}$ holes, necessarily $\dim\langle N\rangle \geq 4$ and therefore $a_0\in\{0,1,3,7\}$.

Lemma~\ref{lem:ps_spec} yields the following linear system of equations for the spectrum of $\mathcal{S}$:
\[
	\begin{pmatrix}
		1 & 1 & 1 & 1 \\
		0 & 1 & 2 & 3 \\
		0 & 0 & 1 & 3
	\end{pmatrix}
	\begin{pmatrix}
		a_0 \\ a_1 \\ a_2 \\ a_3
	\end{pmatrix}
	=
	\begin{pmatrix}
		127 \\
		240 \\
		120
	\end{pmatrix}
\]
Parameterizing by $a_0$, we get the solution
\begin{align*}
	a_1 & = 21 - 3a_0 \\
	a_2 & = 99 + 3a_0 \\
	a_3 & = 7 - a_0
\end{align*}
Plugging in the four possible values $a_0\in\{0,1,3,7\}$ and applying Lemma~\ref{lem:spec2holespec} leads to the following four possibilities.
\[
	\begin{array}{ccc}
	\dim \langle N\rangle & \text{spectrum} & \text{hole spectrum} \\
	\hline
	7 & (1^{21} 2^{99} 3^7) & (3^7 7^{99} 11^{21}) \\
	6 & (0^1 1^{18} 2^{102} 3^6) & (3^3 7^{51} 11^9)\\
	5 & (0^3 1^{12} 2^{108} 3^4) & (3^1 7^{27} 11^3) \\
	4 & (0^7 2^{120}) & (7^{15})
	\end{array}
\]

\subsection{Hole configuration}
\label{subsect:ps17:hole}
The first step for the proof of Theorem~\ref{thm:size16} is the classification of the hole configuration from the above spectra.
The classification will be done entirely by theory, without the need to use a computer.
For $n := \dim\langle N\rangle = 4$ the statement immediately follows from $\#\langle N\rangle = \qbinom{4}{1}{q} = 15 = \#N$.
In the following, we will deal with the cases $n\in\{5,6,7\}$, which get increasingly involved.

\begin{proof}[Proof of Theorem~\ref{thm:size16}, hole structure for $n=5$]
	Let $\dim(\langle N\rangle) = 5$.
	Then the hole spectrum is $(3^1 7^{27} 11^3)$.
	Let $S_1, S_2$ and $S_3$ be the three solids containing $11$ holes.

	We show that the three planes $E_1 = S_1 \cap S_2$, $E_2 = S_2 \cap S_3$ and $E_3 = S_3\cap S_1$ together with $L = S_1 \cap S_2 \cap S_3$ have the claimed properties.

	For any two solids $S_i$ and $S_j$ with $(i,j)\in\{(1,2),(2,3),(3,1)\}$, $\dim(S_i \cap S_j) = 3$, so $h(E_i) = h(S_i \cap S_j) \leq \qbinom{3}{1}{q} = 7$.
	On the other hand by the sieve formula,
	\[
		h(E_i) = h(S_i) + h(S_j) - h(S_i \cup S_j) \geq 11 + 11 - 15 = 7\text{.}
	\]
	So $h(E_i) = 7$, meaning that the plane $E_i$ consists of holes only.

	In the dual geometry of $\PG(N)$, the solids $S_i$ are either collinear or they form a triangle.
	In the first case, $L = E_1 = E_2 = E_3$ consists of holes only and therefore
	\[
		h(S_1 \cup S_2 \cup S_3) = h(S_1 \setminus L) + h(S_2\setminus L) + h(S_3 \setminus L) + h(L) = 3\cdot 4 + 7 > 15\text{,}
	\]
	a contradiction.
	So we are in the second case.
	Here $\dim(L) = 2$, $E_1 + E_2 + E_3 = N$, and the $3\cdot (7-3) + 3 = 15$ points contained in $E_1 \cup E_2 \cup E_3$ are the $15$ holes.
	%Furthermore, we have $\dim(L) = \dim(S_1 \cap S_2 \cap S_3) \in\{2,3\}$.
	%If $\dim(L) = 3$, then 	Hence $\dim(L) = 2$
%
%	Since this intersection is contained in $E_1$ which we have seen to consist entirely of holes, so does $L$ and therefore $h(L) \geq \qbinom{2}{1}{2} = 3$ with equality if and only if $\dim(L) = 2$.
%	On the other hand again by the sieve formula
%	\begin{multline*}
%		h(L)
%		= -h(S_1) -h(S_2) -h(S_3) + h(E_1) + h(E_2) + h(E_2) + h(S_1 \cup S_2 \cup S_3) \\
%		\leq -11 -11 -11 + 7 + 7 + 7 + 15 = 3\text{.}
%	\end{multline*}
%	Hence $\dim(L) = 2$.
%	Furthermore 
%	\begin{align*}
%		& \phantom{{} = {}}\dim(E_1 + E_2 + E_3) \\
%		& = \dim(E_1) + \dim(E_2) + \dim(E_3) \\
%		& \phantom{{} = {}} - \dim(E_1 \cap E_2) - \dim(E_2 \cap E_3) - \dim(E_3\cap E_1) + \dim(E_1 \cap E_2 \cap E_3) \\
%		& = \dim(E_1) + \dim(E_2) + \dim(E_3) - \dim(L) - \dim(L) - \dim(L) + \dim(L) \\
%		& = 5\text{,}
%	\end{align*}
%	so $\langle E_1, E_2, E_3\rangle = \langle N\rangle$.
\end{proof}

The following counting method is a direct consequence of the sieve formula and will be used several times.
\begin{lemma}
	\label{lem:sieve}
	Let $W \leq Y \leq V$ with $\dim(Y/W) = 2$.
	Let $X_1, X_2, X_3$ be the three intermediate spaces of $W \leq Y$ with $\dim(Y/X_i) = 1$.
	Then
	\[
		2h(W) = h(X_1) + h(X_2) + h(X_3) - h(Y)\text{.}
	\]
\end{lemma}

We are going to apply Lemma~\ref{lem:sieve} to $Y = \langle N\rangle$.
Then $h(\langle N\rangle) = 15$ and $h(X_i) \in\{3, 7, 11\}$.

\begin{lemma}
	\label{lem:N_HE_odd}
	For each $W \in\qbinom{\langle N\rangle}{n-2}{2}$, $h(W)$ is odd.
\end{lemma}

\begin{proof}
	Denoting the three codimension $1$ intermediate spaces of $W \leq \langle N\rangle$ by $X_1, X_2, X_3$, Lemma~\ref{lem:sieve} yields
	\[
		2h(W) = h(X_1) + h(X_2) + h(X_3) - h(\langle N\rangle)
		\equiv -1 -1 -1 +1 \equiv 2\mod 4
	\]
\end{proof}

As a more detailed analysis, all the possibilities for $(h(X_1),h(X_2),h(X_3))$ (up to permutations) allowed by Lemma~\ref{lem:sieve} are listed in Table~\ref{tbl:eier}.
In particular, the distributions $(7,3,3)$ and $(3,3,3)$ are not possible as we get the negative numbers $h(W) = -1$ and $h(W) = -3$, respectively.
Moreover, $(11,11,3)$ is not possible as the resulting $h(W) = 5$ contradicts $W \leq X_3$.
Table~\ref{tbl:eier} also shows that in the cases $h(W) \in \{5,7,9\}$ the numerical hole distributions are unique.

\begin{table}
\caption{Codimension $2$ hole distributions in $\langle N\rangle$}
\label{tbl:eier}
\centering
$\begin{array}{ccc|c}
h(X_1) & h(X_2) & h(X_3) & h(W) \\
\hline
11 & 11 & 11 & 9 \\
11 & 11 & 7 & 7 \\
11 & 7 & 7 & 5 \\
11 & 7 & 3 & 3 \\
7 & 7 & 7 & 3 \\
11 & 3 & 3 & 1 \\
7 & 7 & 3 & 1
\end{array}$
\end{table}

\begin{proof}[Proof of Theorem~\ref{thm:size16}, hole structure for $n=6$]
	By the hole spectrum $(3^3 7^{51} 11^9)$, $\langle N\rangle$ contains three $4$-flats of multiplicity $3$.
	We denote their intersection by $E$.
	By Table~\ref{tbl:eier}, $\dim(E) \neq 4$.
	So $\dim(E) = 3$ and the factor geometry $\langle N\rangle / E$ carries the structure of a Fano plane.
	Therefore, we may label the seven intermediate solids of $E < \langle N\rangle$ by $S_{\vek{x}}$ and the seven intermediate $4$-flats by $F_{\vek{x}}$ with $\vek{x}\in\F_2^3\setminus\{\vek{0}\}$, such that $S_{\vek{x}} \leq F_{\vek{y}}$ if and only if $\vek{x} \perp \vek{y}$.
	Furthermore, we may assume that $h(F_{100}) = h(F_{010} = h(F_{001}) = 3$.
	Applying Table~\ref{tbl:eier} to $W = F_{010} \cap F_{001} = S_{001}$ and $\{X_1,X_2,X_3\} = \{F_{010},F_{001},F_{011}\}$, we get $h(S_{001}) = 1$ and $h(F_{011}) = 11$.
	
	By Table~\ref{tbl:eier}, $h(S_{100}) = h(S_{010}) = h(S_{001}) = 1$, $h(F_{011}) = h(F_{101}) = h(F_{110}) = 11$ and $h(S_{111}) = 9$.%
	\footnote{
	Example: $W = S_{001}$ is of codimension $2$ in $\langle N\rangle$.
	The three intermediate subspaces of $W < \langle N\rangle$ of codimension $1$ are given by $\{X_1,X_2,X_3\} = \{F_{010},F_{001},F_{011}\}$.
	As $h(F_{010}) = h(F_{001}) = 3$, the second last line of Table~\ref{tbl:eier} is the only possibility.
	Therefore, $h(F_{001}) = 11$ and $h(S_{001}) = 1$.
	}
	From $S_{011} \leq F_{100}$, we get $h(S_{011}) \leq h(F_{100}) = 3$ and therefore $h(S_{011})\in\{1,3\}$.
	The application of Lemma~\ref{lem:sieve} to $E \leq F_{100}$ gives $h(S_{011}) = 2h(E) + 1$, so $h(E) \in \{0,1\}$.
	Doing the same for $S_{101}$ and $S_{110}$, we arrive at
	\[
		h(S_{011}) = h(S_{101}) = h(S_{110}) =  2h(E) + 1\text{.}
	\]
	If $h(E) = 0$, the three solids in $F_{111}$ containing $E$ are of multiplicity $1$ and therefore by Table~\ref{tbl:eier}, $h(F_{111}) = 3$.
	This contradicts the hole spectrum, as $F_{111}$ would be a fourth $4$-flat of multiplicity $3$.
	So $h(E) = 1$, $h(S_{011}) = h(S_{101}) = h(S_{110}) = 3$ and $h(F_{111}) = 7$.
	We denote the single hole in $E$ by $P$.

	Let $A$ be the affine solid $S_{111} \setminus E$.
	From $h(A) = h(S_{111}) - h(E) = 8$, we see that all the points in $A$ are holes.
	We look at the chain $A \leq S_{111} \leq F_{100}$.
	Out of the $11$ holes in $F_{100}$, $8$ are contained in $A$, the single hole $P$ is contained in $S_{111}\setminus A = E$ and $2$ further holes are contained in $F_{100}\setminus S_{111}$.
	The three holes in $F_{100} \setminus A$ must be collinear:
	Otherwise there is a solid $S$ of $F_{100}$ containing $2$ of these $3$ holes, and as $S \cap A$ is an affine plane, $h(S) = 4 + 2 = 6$ which is not possible by Lemma~\ref{lem:N_HE_odd}.
	Repeating the argument for $F_{010}$ and $F_{001}$, we get that $N \setminus A = L_1 \cup L_2 \cup L_3$, where the $L_i$ are lines passing through the common point $P$.
	Now from
	\begin{multline*}
		6 
		= \dim\langle N\rangle
		= \dim(\langle A\rangle + \langle N\setminus A\rangle) \\
		= \dim S_{111} + \dim\langle N\setminus A\rangle - \dim(\underbrace{S_{111} \cap \langle N\setminus A\rangle}_{=P})
		= 4 + \dim\langle N\setminus A\rangle - 1
	\end{multline*}
	we get $\dim\langle N\setminus A\rangle = 3$.
	So $N\setminus A$ is a plane and all its $7$ points are holes.
	% TODO: evtl Bild
\end{proof}

Now we partially extend the analysis of Table~\ref{tbl:eier} to $\dim(\langle N\rangle/W) = 3$, characterizing the hole distribution to the intermediate lattice of $W \leq \langle N\rangle$ for \enquote{heavy} subspaces $W$.
While this information is only needed for the last case $n=7$, the proof works for any value of $n$.

\begin{lemma}
	\label{lem:fanoeier}
	Let $W \leq \langle N\rangle$ of codimension $3$ and $h(W) \geq 6$.
	Denoting the set of the seven intermediate spaces of codimension $1$ by $\mathcal{X}$ and of the seven intermediate spaces of codimension $2$ by $\mathcal{Y}$, one of the following two cases arises:
	\begin{enumerate}[(i)]
		\item\label{lem:fanoeier:w8} $h(W) = 8$, $h(X) = 9$ for all $X\in\mathcal{X}$ and $h(Y) = 11$ for all $Y\in\mathcal{Y}$.
		\item\label{lem:fanoeier:w7} $h(W) = 7$. There is a single $Y\in\mathcal{Y}$ of multiplicity $7$, and the six remaining subspaces in $\mathcal{Y}$ are of multiplicity $11$.
		The three $X\in\mathcal{X}$ contained in $Y$ are of multiplicity $7$, the other $4$ subspaces in $\mathcal{X}$ are of multiplicity $9$.
	\end{enumerate}
	In particular, $h(W) = 6$ is not possible.
\end{lemma}

\begin{proof}
	For all $X\in\mathcal{X}$, $W \leq X$, so $h(X) \geq h(W) \geq 6$.
	So only the first two lines in Table~\ref{tbl:eier} are possible, and in particular $h(X)\in\{7,9\}$ for all $X\in\mathcal{X}$ and $h(Y)\in\{7,11\}$ for all $Y\in\mathcal{Y}$.

	The intermediate lattice of $W \leq \langle N\rangle$ carries the structure of a Fano plane.
	If there are two distinct $Y_1, Y_2\in\mathcal{Y}$ of multiplicity $7$, Table~\ref{tbl:eier} shows that $Y_1 \cap Y_2 \in \mathcal{X}$ is of multiplicity at most $5$, which is a contradiction.
	So the number of $Y\in\mathcal{Y}$ of multiplicity $7$ is either $0$ or $1$.
	After several applications of Lemma~\ref{lem:sieve} and Table~\ref{tbl:eier}, these two possibilities are completed to the stated cases.
\end{proof}

\begin{lemma}
	\label{lem:noskewlines}
	Let $\mathcal{L}$ be a set of lines in some projective geometry such that no pair of lines in $\mathcal{L}$ is skew.
	Then at least one of the following statements is true:
	\begin{enumerate}[(i)]
		\item All the lines in $\mathcal{L}$ pass through a common point.
		\item The lines in $\mathcal{L}$ are contained in a common plane.
	\end{enumerate}
\end{lemma}

\begin{proof}
	Assume that there is no common point of the lines in $\mathcal{L}$.
	Then there exist three lines $L_1,L_2,L_3\in\mathcal{L}$ forming a triangle.
	Let $E$ be the plane spanned by $L_1$, $L_2$ and $L_3$.
	Let $L'\in\mathcal{L}\setminus\{L_1,L_2,L_3\}$.
	Then $P_i := L' \cap L_i\in E$ for all $i\in\{1,2,3\}$, and as the three lines $L_1, L_2$ and $L_3$ do not pass through a common point, $\#\{P_1,P_2,P_3\} \geq 2$.
	This implies $L' \leq E$.
\end{proof}

The main step to the classification is the following lemma.

\begin{lemma}
	\label{lem:size16:n7:nosqew}
	For $n=7$, the hole set $N$ contains exactly $7$ lines.
	No pair of these lines is skew.
\end{lemma}

\begin{proof}
	Let $\ell$ be the number of lines contained in $N$.
	We count the set $X$ of pairs $(H,\{P_1,P_2,P_3\}) \in\qbinom{V}{6}{2} \times \binom{N}{3}$ with $\{P_1,P_2,P_3\} \subseteq H$ in two ways.
	By the hole spectrum $(3^7 7^{99} 11^{21})$,
	\[
		\#X = 7\cdot\binom{3}{3} + 99\cdot\binom{7}{3} + 21\cdot\binom{11}{3} = 6937\text{.}
	\]
	On the other hand, each of the $\ell$ collinear triples of holes generates a $2$-subspace, and each of the $\binom{15}{3} - \ell$ non-collinear triples of holes generates a $3$-subspace, showing that
	\[
		\#X = \ell\cdot \qbinom{7-2}{6-2}{2} + \left(\binom{15}{3} - \ell\right) \qbinom{7-3}{6-3}{2} = 16\ell + 6825\text{.}
	\]
	Thus $16\ell + 6825 = 6937$ and hence $N$ contains exactly $\ell = 7$ lines.
	Let $\mathcal{L}$ be the set of these lines.

	Assume that $L_1,L_2\in\mathcal{L}$ are skew.
	Then $S := \langle L_1, L_2\rangle$ is a solid.
	By Lemma~\ref{lem:fanoeier}, there are two possible cases.

	\paragraph{Case 1}
	$h(S) = 7$ and $S$ is contained in a hyperplane $H$ of multiplicity $7$.
	Let $P$ be the hole of $S$ which is not covered by the lines $L_1$ and $L_2$.
	Let $E$ be a plane of $S$ passing through $L_1$ and not containing $P$.
	As $E$ intersects $L_2$ in a point, $h(E) = 3 + 1 = 4$.
	Since there is no hole in $H\setminus S$, any $4$-flat $F \leq H$ with $F \cap S = E$ is of multiplicity $h(F) = h(E) = 4$.
	This is a contradiction to Lemma~\ref{lem:N_HE_odd}.

	\paragraph{Case 2}
	$h(S) = 8$.
	So apart from the points on $L_1$ and $L_2$, $S$ contains two further holes $P_1$ and $P_2$.
	Each of the points $P_i$ ($i\in\{1,2\}$) is contained in a single line of $N \cap S$, which is given by the line passing through $P_i$ and the intersection point of $L_2$ and the plane spanned by $P_i$ and $L_1$.
	So in total, $S$ contains $4$ lines of $\mathcal{L}$.

	By $\#\mathcal{L} = 7$ there exists a line $L_3\in\mathcal{L}$ not contained in $S$, implying that $L_3$ contains at least $2$ holes not contained in $S$.
	Let $H = \langle L_1, L_2, L_3\rangle$.
	Then $h(H) \geq h(S) + 2 = 10$.
	By Table~\ref{tbl:eier}, any $4$-flat contains at most $9$ holes, so $\dim(H) = 6$, $h(H) = 11$ and by the dimension formula
	\[
		\dim(S\cap L_3) = \dim(S) + \dim(L_3) - \dim(S + L_3) = 4 + 2 - \dim(H) = 0\text{.}
	\]
	Therefore, any two of the three lines $L_1$, $L_2$ and $L_3$ are skew.

	By the already excluded Case~1, the solids $S' = \langle L_1, L_3\rangle$ and $S'' = \langle L_2, L_3\rangle$ are of multiplicity $h(S') = h(S'') = 8$.
	So each of the solids $S$, $S'$ and $S''$ contains two extra holes which are not contained in $L_1$, $L_2$ or $L_3$.
	As the $11$ holes in $H$ are already given by $P_1$, $P_2$ and the points on $L_1$, $L_2$ and $L_3$, these two extra holes must be $P_1$, $P_2$ for all three solids $S$, $S'$ and $S''$.
	Now $S \cap S'$ contains at least the five holes given by $P_1$, $P_2$ and the holes on $L_1$.
	Therefore, $\dim(S\cap S') \geq 3$.
	However, the dimension formula yields
	\[
		\dim(S\cap S') = \dim(S) + \dim(S') - \dim(S+S') = 4 + 4 - \dim(H) = 2\text{.}
	\]
	Contradiction.
\end{proof}

\begin{proof}[Proof of Theorem~\ref{thm:size16}, hole structure for $n=7$]
	By Lemma~\ref{lem:size16:n7:nosqew} and Lemma~\ref{lem:noskewlines}, we are in one of the following cases:
	\begin{enumerate}[(i)]
		\item
			The lines in $\mathcal{L}$ pass through a common point $P$.
			Thus, $\mathcal{L}$ covers all the $1 + 7\cdot 2 = 15$ holes, showing that $N = \bigcup \mathcal{L}$.
			In particular, $\dim\langle\bigcup\mathcal{L}\rangle = n = 7$.
			So for any set $\mathcal{L}'$ of $6$ lines in $\mathcal{L}$, $\dim\langle\bigcup\mathcal{L}'\rangle \in\{6,7\}$.
			Indeed, $\dim\langle\bigcup\mathcal{L}'\rangle = 7$, since otherwise $\langle \bigcup\mathcal{L}'\rangle$ is a hyperplane containing at least $13$ holes which contradicts the hole spectrum.
			So $\{L/P \mid L \in \mathcal{L}\}$ is a projective basis of the factor geometry $\PG(V/P)$.
		\item
			There is a plane $E\in\qbinom{V}{3}{2}$ such that $\mathcal{L} = \qbinom{E}{2}{2}$.
			By $E\subseteq N$, $\mathcal{S} \cup \{E\}$ is a partial plane spread of maximum size $17$.
			By Lemma~\ref{lem:ps17_holes}, its hole set $N\setminus E$ is an affine solid, showing that $N$ is the disjoint union of a plane and an affine solid.
	\end{enumerate}
\end{proof}

\subsection{Extendible partial plane spreads of size $16$}
For the classification of the extendible partial plane spreads of size
$16$, we make use of the classification of the maximum partial plane
spreads of Section~\ref{sect:ps17}.

A partial plane spread $\mathcal{S}$ of size $16$ is extendible to
size $17$ if and only if its hole set contains a plane $E$.  In this
case, $N = E \cup A$ where $A$ is an affine solid,
$\hat{\mathcal{S}} = \mathcal{S} \cup\{E\}$ is a maximum partial plane
spread and its hole set is $A$.  By the dimension formula,
$\dim(E \cap \langle A\rangle) = 7-n$, showing that $7-n$ must appear
in the type of $\hat{\mathcal{S}}$.

The remaining possibilities are:
	\begin{enumerate}[(i)]
		\item
		If $n=4$, the hole set $N$ contains $15$ planes.
		Extending $\mathcal{S}$ by any of these planes $E$ leads to a maximum partial plane spread $\hat{\mathcal{S}}$ of type $(3^1)$, and $E$ is the unique block contained in the hole space of $\hat{\mathcal{S}}$.
		\item 
		If $n=5$, the hole set $N$ contains $3$ planes.
		Extending $\mathcal{S}$ by any of these planes $E$ leads to a maximum partial plane spread $\hat{\mathcal{S}}$ of type $(2^1 1^4)$, and $E$ is the unique block intersecting the hole space of $\hat{\mathcal{S}}$ in a line.
		\item
		For $n=6$, the hole set $N$ contains a single plane.
		The resulting maximum partial plane spread $\hat{\mathcal{S}}$ is of type $(2^1 1^4)$ or $(1^7)$.
		\item 
		For $n=7$, the hole set $N$ contains a single plane.
		For the resulting maximum partial plane spread $\hat{\mathcal{S}}$, all three types $(3^1)$, $(2^1 1^4)$ and $(1^7)$ are possible.
\end{enumerate}

\begin{remark}
In the cases $n\in\{4,5\}$, the block $E$ of a maximum partial plane spread $\hat{\mathcal{S}}$ is called \emph{moving} as it can be exchanged for any of the $15$ or $3$ other planes preserving the property of $\hat{\mathcal{S}}$ being a maximum partial spread.
\end{remark}

Now for the generation of partial plane spreads of size $16$ with $n = 4$, we go through the $150$ maximum partial plane spreads $\hat{\mathcal{S}}$ of type $(3^1)$ and remove the block in the hole space.
It is possible that non-isomorphic maximum partial plane spreads yield isomorphic reductions.
More precisely, as the hole set of the reduction contains $15$ planes, an isomorphism type might be generated up to $15$ times in this way.
The actual distribution of numbers of reductions falling together is
\[
	(1^1 2^9 3^8 4^6 5^6 6^2 7^2 9^3)\text{.}
\]
So in total, there are $37$ isomorphism types of partial plane spreads with $n=4$.
In Theorem~\ref{thm:mrd} and Section~\ref{sec:ssc:3_16__4_1}, we will see that these partial plane spreads correspond to vector space partitions of $\PG(6,2)$ of type $(3^{16} 4^1)$, to binary $3\times 4$ MRD codes of minimum rank distance $3$ (which are given explicitly in Theorem~\ref{thm:mrd}) and to $(7,17,6)_2$ subspace codes of dimension distribution $(3^{16} 4^1)$.

Similarly, for the generation of partial plane spreads of size $16$ with $n = 5$, we go through the $180$ maximum partial plane spreads $\hat{\mathcal{S}}$ of type $(2^1 1^4)$ and remove the block intersecting the hole space in dimension $2$.
Since the hole set of the reduction contains $3$ planes, an isomorphism type may be generated up to three times.
The $180$ starting spreads produced the distribution $(1^{10} 2^7 3^{52})$ of reductions falling together.

For $n\in\{6,7\}$, we proceed in a similar manner, starting with the maximum partial plane spreads of types $(2^1 1^4)$ and $(1^7)$, or of all three types, respectively.
The resulting numbers of isomorphism types of extendible partial plane spreads of size $16$ are summarized in Table~\ref{tbl:extendible16}.

\begin{table}
	\caption{Types of extendible partial plane spreads in $\PG(6,2)$ of size $16$}
	\label{tbl:extendible16}
	\centering
	\vspace{2mm}
	$\begin{array}{c|ccc|c}
		 & \hat{\mathcal{S}}\text{ of type } (3^1) & \hat{\mathcal{S}}\text{ of type } (2^1 1^4) & \hat{\mathcal{S}}\text{ of type } (1^7) & \Sigma \\
		\hline
		n=4 & 37 & 0 & 0 & 37 \\
		n=5 & 0 & 69 & 0 & 69 \\
		n=6 & 0 & 604 & 2689 & 3293 \\
		n=7 & 1324 & 1890 & 3844 & 7058 \\
		\hline
		\Sigma & 1361 & 2563 & 6533 & 10457
	\end{array}$
\end{table}

% n=6:
%    {* 0^^10, 1^^6 *}^^2689,
%    {* 0^^12, 1^^3, 2 *}^^604

\subsection{Complete partial plane spreads of size $16$}
For the remaining classification of the complete partial plane spreads, we fix the hole set $N$ as determined above and let $\mathcal{S}$ be a partial plane spread of size $16$ having hole set $N$.
Furthermore, we fix a hyperplane $H$ containing $3$ holes and compute the stabilizer $G$ of $N \cup \{H\}$ in $\GL(7,2)$ as a group of order $46080$.

By Lemma~\ref{lem:ps_hyperplane_holes}, $H$ contains $3$ blocks of $\mathcal{S}$.
Under the action of $G$, we find $3$ possibilities to add a single plane in $H$, $18$ possibilities to add two disjoint planes and $275$ possibilities to add three pairwise disjoint planes in $H$.
For each of these $275$ configurations, we enumerate all possibilities for the extension to a partial plane spread of size $16$ attaining the hole set $N$.
Stating this problem as an exact cover problem, the running time per starting configuration is only a few seconds.
The number of extensions per starting configuration ranges between $88$ and $2704$.
In total, we get $66490$ extensions.
Filtering out isomorphic copies, we end up with $3988$ isomorphism types.

\begin{example}
	We give the most symmetric complete partial plane spreads $\mathcal{S}_1$ and $\mathcal{S}_2$ as lists of generator matrices in Tables~\ref{table:s16complete_1} and~\ref{table:s16complete_2}, both having an automorphism group of order $24$ isomorphic to the symmetric group $S_4$.
	In both cases, the holes are given by the seven lines passing through a unit vector and the all-one vector.
	The orbit structure of $\mathcal{S}_1$ is $\{1\}\{2\}\{3,4,5,6,7,8\}\{9,10,11,12,13,14,15,16\}$, where $1$ denotes the first matrix in the list, $2$ denotes the second matrix in the list and so on.
	The orbit structure of $\mathcal{S}_2$ is $\{1\}\{2,3,4\}\{5,6,7,8\}\{9,10,11,12,13,14,15,16\}$.

\begin{table}
\caption{The complete partial plane spread $\mathcal{S}_1$ of size $16$}
\label{table:s16complete_1}
\vspace{2mm}
\centering$\setlength{\arraycolsep}{2pt}
\begin{array}{cccc}
    \left(\begin{smallmatrix}
    1 & 0 & 0 & 0 & 0 & 0 & 1 \\
    0 & 1 & 0 & 0 & 0 & 0 & 1 \\
    0 & 0 & 1 & 0 & 0 & 0 & 1
    \end{smallmatrix}\right)\text{,} &
    \left(\begin{smallmatrix}
    1 & 0 & 0 & 0 & 0 & 1 & 1 \\
    0 & 0 & 1 & 0 & 1 & 1 & 1 \\
    0 & 0 & 0 & 1 & 1 & 1 & 0
    \end{smallmatrix}\right)\text{,} &
    \left(\begin{smallmatrix}
    0 & 1 & 0 & 0 & 1 & 0 & 1 \\
    0 & 0 & 1 & 0 & 0 & 1 & 0 \\
    0 & 0 & 0 & 1 & 0 & 1 & 0
    \end{smallmatrix}\right)\text{,} &
    \left(\begin{smallmatrix}
    1 & 0 & 1 & 0 & 1 & 1 & 1 \\
    0 & 1 & 0 & 0 & 0 & 1 & 0 \\
    0 & 0 & 0 & 1 & 0 & 0 & 1
    \end{smallmatrix}\right)\text{,} \\[+3mm]
    \left(\begin{smallmatrix}
    1 & 0 & 0 & 0 & 0 & 1 & 0 \\
    0 & 1 & 0 & 0 & 1 & 0 & 0 \\
    0 & 0 & 1 & 1 & 1 & 1 & 1
    \end{smallmatrix}\right)\text{,} &
    \left(\begin{smallmatrix}
    1 & 1 & 0 & 1 & 0 & 0 & 0 \\
    0 & 0 & 0 & 0 & 1 & 0 & 1 \\
    0 & 0 & 0 & 0 & 0 & 1 & 1
    \end{smallmatrix}\right)\text{,} &
    \left(\begin{smallmatrix}
    1 & 0 & 0 & 0 & 1 & 0 & 0 \\
    0 & 1 & 1 & 0 & 0 & 1 & 0 \\
    0 & 0 & 0 & 1 & 1 & 0 & 0
    \end{smallmatrix}\right)\text{,} &
    \left(\begin{smallmatrix}
    1 & 0 & 0 & 1 & 1 & 1 & 1 \\
    0 & 1 & 0 & 1 & 0 & 0 & 0 \\
    0 & 0 & 1 & 0 & 1 & 0 & 0
    \end{smallmatrix}\right)\text{,} \\[+3mm]
    \left(\begin{smallmatrix}
    1 & 0 & 0 & 1 & 1 & 0 & 0 \\
    0 & 1 & 0 & 1 & 1 & 0 & 1 \\
    0 & 0 & 1 & 0 & 1 & 0 & 1
    \end{smallmatrix}\right)\text{,} &
    \left(\begin{smallmatrix}
    1 & 0 & 0 & 0 & 1 & 1 & 1 \\
    0 & 0 & 1 & 0 & 1 & 1 & 0 \\
    0 & 0 & 0 & 1 & 1 & 0 & 1
    \end{smallmatrix}\right)\text{,} &
    \left(\begin{smallmatrix}
    1 & 0 & 0 & 0 & 1 & 0 & 1 \\
    0 & 1 & 0 & 0 & 1 & 1 & 1 \\
    0 & 0 & 0 & 1 & 0 & 1 & 1
    \end{smallmatrix}\right)\text{,} &
    \left(\begin{smallmatrix}
    0 & 1 & 0 & 0 & 1 & 1 & 0 \\
    0 & 0 & 1 & 0 & 0 & 1 & 1 \\
    0 & 0 & 0 & 1 & 1 & 1 & 1
    \end{smallmatrix}\right)\text{,} \\[+3mm]
    \left(\begin{smallmatrix}
    1 & 0 & 0 & 1 & 0 & 0 & 1 \\
    0 & 1 & 0 & 0 & 0 & 1 & 1 \\
    0 & 0 & 1 & 1 & 0 & 1 & 0
    \end{smallmatrix}\right)\text{,} &
    \left(\begin{smallmatrix}
    1 & 0 & 1 & 0 & 0 & 1 & 0 \\
    0 & 1 & 1 & 0 & 0 & 0 & 1 \\
    0 & 0 & 0 & 0 & 1 & 1 & 1
    \end{smallmatrix}\right)\text{,} &
    \left(\begin{smallmatrix}
    1 & 0 & 0 & 1 & 0 & 1 & 1 \\
    0 & 1 & 0 & 1 & 1 & 1 & 0 \\
    0 & 0 & 1 & 1 & 1 & 0 & 1
    \end{smallmatrix}\right)\text{,} &
    \left(\begin{smallmatrix}
    1 & 0 & 0 & 0 & 1 & 1 & 0 \\
    0 & 1 & 0 & 1 & 0 & 1 & 0 \\
    0 & 0 & 1 & 1 & 1 & 1 & 0
    \end{smallmatrix}\right)\phantom{,}
\end{array}$
\end{table}

\begin{table}
\caption{The complete partial plane spread $\mathcal{S}_2$ of size $16$}
\label{table:s16complete_2}
\vspace{2mm}
\centering$\setlength{\arraycolsep}{2pt}
\begin{array}{cccc}
    \left(\begin{smallmatrix}
    1 & 0 & 0 & 0 & 0 & 1 & 1 \\
    0 & 1 & 0 & 1 & 0 & 1 & 1 \\
    0 & 0 & 1 & 1 & 0 & 0 & 1
    \end{smallmatrix}\right)\text{,} &
    \left(\begin{smallmatrix}
    0 & 1 & 0 & 0 & 1 & 0 & 1 \\
    0 & 0 & 1 & 0 & 0 & 1 & 0 \\
    0 & 0 & 0 & 1 & 0 & 1 & 0
    \end{smallmatrix}\right)\text{,} &
    \left(\begin{smallmatrix}
    1 & 0 & 0 & 0 & 0 & 0 & 1 \\
    0 & 0 & 0 & 1 & 0 & 0 & 1 \\
    0 & 0 & 0 & 0 & 1 & 0 & 1
    \end{smallmatrix}\right)\text{,} &
    \left(\begin{smallmatrix}
    1 & 1 & 0 & 1 & 0 & 1 & 1 \\
    0 & 0 & 1 & 0 & 0 & 0 & 1 \\
    0 & 0 & 0 & 0 & 1 & 1 & 0
    \end{smallmatrix}\right)\text{,} \\[+3mm]
    \left(\begin{smallmatrix}
    1 & 0 & 0 & 1 & 1 & 1 & 1 \\
    0 & 1 & 0 & 1 & 0 & 0 & 0 \\
    0 & 0 & 1 & 0 & 1 & 0 & 0
    \end{smallmatrix}\right)\text{,} &
    \left(\begin{smallmatrix}
    1 & 0 & 0 & 0 & 0 & 1 & 0 \\
    0 & 1 & 0 & 0 & 1 & 0 & 0 \\
    0 & 0 & 1 & 1 & 1 & 1 & 1
    \end{smallmatrix}\right)\text{,} &
    \left(\begin{smallmatrix}
    1 & 0 & 1 & 0 & 0 & 0 & 0 \\
    0 & 1 & 1 & 0 & 0 & 0 & 0 \\
    0 & 0 & 0 & 1 & 1 & 1 & 0
    \end{smallmatrix}\right)\text{,} &
    \left(\begin{smallmatrix}
    1 & 0 & 1 & 0 & 1 & 0 & 0 \\
    0 & 1 & 0 & 0 & 0 & 0 & 1 \\
    0 & 0 & 0 & 0 & 0 & 1 & 1
    \end{smallmatrix}\right)\text{,} \\[+3mm]
    \left(\begin{smallmatrix}
    0 & 1 & 0 & 0 & 1 & 1 & 1 \\
    0 & 0 & 1 & 0 & 1 & 1 & 0 \\
    0 & 0 & 0 & 1 & 0 & 1 & 1
    \end{smallmatrix}\right)\text{,} &
    \left(\begin{smallmatrix}
    1 & 0 & 0 & 0 & 1 & 0 & 1 \\
    0 & 1 & 0 & 0 & 1 & 1 & 0 \\
    0 & 0 & 0 & 1 & 1 & 1 & 1
    \end{smallmatrix}\right)\text{,} &
    \left(\begin{smallmatrix}
    1 & 0 & 0 & 0 & 1 & 1 & 1 \\
    0 & 1 & 0 & 1 & 1 & 0 & 1 \\
    0 & 0 & 1 & 1 & 1 & 1 & 0
    \end{smallmatrix}\right)\text{,} &
    \left(\begin{smallmatrix}
    1 & 0 & 1 & 0 & 0 & 0 & 1 \\
    0 & 1 & 1 & 0 & 1 & 0 & 1 \\
    0 & 0 & 0 & 1 & 1 & 0 & 1
    \end{smallmatrix}\right)\text{,} \\[+3mm]
    \left(\begin{smallmatrix}
    1 & 0 & 0 & 1 & 0 & 0 & 1 \\
    0 & 0 & 1 & 1 & 0 & 1 & 1 \\
    0 & 0 & 0 & 0 & 1 & 1 & 1
    \end{smallmatrix}\right)\text{,} &
    \left(\begin{smallmatrix}
    1 & 0 & 0 & 1 & 0 & 1 & 1 \\
    0 & 1 & 0 & 1 & 0 & 1 & 0 \\
    0 & 0 & 1 & 0 & 0 & 1 & 1
    \end{smallmatrix}\right)\text{,} &
    \left(\begin{smallmatrix}
    1 & 0 & 0 & 0 & 1 & 1 & 0 \\
    0 & 1 & 0 & 0 & 0 & 1 & 1 \\
    0 & 0 & 1 & 0 & 1 & 0 & 1
    \end{smallmatrix}\right)\text{,} &
    \left(\begin{smallmatrix}
    1 & 0 & 0 & 1 & 1 & 0 & 0 \\
    0 & 1 & 0 & 1 & 1 & 1 & 0 \\
    0 & 0 & 1 & 1 & 0 & 1 & 0
    \end{smallmatrix}\right)\phantom{.}
\end{array}$
\end{table}
\end{example}

\section{MRD codes}
\label{sect:mrd}
In this section, we fix the following notation:
Let $m,n$ be positive integers.
We define $\Phi : \Inn(m,n,q) \to \PGammaL(m+n,q)$, mapping an inner automorphism $\phi : A \mapsto P\sigma(A)Q + R$ to
\[
	\Phi(\phi) : \langle\vek{x}\rangle \mapsto
	\left\langle \sigma(\vek{x}) \begin{pmatrix} P^{-1} & P^{-1} R \\ 0 & Q \end{pmatrix}\right\rangle\text{.}
\]
It is easily checked that $\Phi$ is an injective group homomorphism.
The image of $\Phi$ consists of all automorphisms of $\PG(m+n-1,q)$ fixing the span $S$ of the last $n$ unit vectors.

For all $\phi\in\Inn(m,n,q)$ and all $A\in\F_q^{m\times n}$, we have
\begin{align*}
	(\Phi(\phi) \circ \Lambda)(A)
	& =
	\Phi(\phi)(\langle(I_m \mid A)\rangle) \\
	& = 
	\left\langle \sigma(I_m \mid A) \begin{pmatrix}
	P^{-1} & P^{-1} R \\
	0 & Q
	\end{pmatrix}\right\rangle \\
	& = 
	\left\langle (I_m \mid \sigma(A)) \begin{pmatrix}
	P^{-1} & P^{-1} R \\
	0 & Q
	\end{pmatrix}\right\rangle \\
	& = \langle (P^{-1} \mid P^{-1} R + \sigma(A)Q)\rangle \\
	& = \langle (I_m \mid R + P\sigma(A)Q)\rangle \\
	& = (\Lambda\circ\phi)(A)\text{.}
\end{align*}
So $\Lambda \circ \phi = \Phi(\phi) \circ \Lambda$.

%\begin{lemma}
%	Let $\mathcal{C} \subseteq \qbinom{\F_q^v}{k}{q}$ such that there exists an $S\in\qbinom{\F_q^v}{v-k}{q}$ with $B \cap S = \{\vek{0}\}$ for all $B\in\mathcal{C}$.
%	Then there exists a rank code $\mathcal{D}\subseteq\F_q^{k\times v-k}$ such that $\mathcal{C}$ is linearly isomorphic to $\Lambda(\mathcal{D})$.
%\end{lemma}
%
%\begin{proof}
%	There is a linear automorphism $\phi$ of $\PG(v-1,k)$ mapping $S$ to the span $S'$ of the last $v-k$ unit vectors.
%	Let $B\in\mathcal{C}$.
%	As $\phi(B')$ has trivial intersection with $S'$, its generator matrix in row reduced echelon form has the structure $(I_m \mid A_B)$ with $A_B\in\F_q^{k\times v-k}$.
%	Setting $\mathcal{D} = \{A_B \mid B\in\phi(\mathcal{C})\}$, we have $\Lambda(\mathcal{D}) = \phi(\mathcal{C})$.
%\end{proof}

\begin{lemma}
	\begin{enumerate}[(a)]
		\item Let $\mathcal{C}$ be an $m\times n$ MRD code over $\F_q$.
			Then the inner automorphism group of $\mathcal{C}$ is given by $\Phi^{-1}(\Aut(\Lambda(\mathcal{C})))$.
		\item Let $\mathcal{C}$ and $\mathcal{C}'$ be two $m\times n$ MRD codes over $\F_q$.
			Then $\mathcal{C}$ and $\mathcal{C}'$ are isomorphic as rank metric codes under an inner isomorphism if and only if $\Lambda(\mathcal{C})$ and $\Lambda(\mathcal{C}')$ are isomorphic as subspace codes.
	\end{enumerate}
\end{lemma}

\begin{proof}
	By Fact~\ref{fct:mrd_structure}, for any lifted MRD code $\hat{\mathcal{C}}$, the span of the last $n$ unit vectors is the unique subspace $S$ of $\F_q^{m+n}$ of dimension $n$ such that $B \cap S$ is trivial for all $B\in \hat{\mathcal{C}}$.
	Therefore, any $\psi\in\PGammaL(m+n,q)$ mapping a lifted MRD code to another one (or the same) has the form $\Phi(\phi)$ with $\phi\in\Inn(m,n,q)$.
	The proof is finished using the above statements about $\Phi$.
\end{proof}

By the above lemma, instead of classifying lifted $m\times n$ MRD codes with $m\leq n$ of minimum rank distance $d$ up to inner automorphisms, we can classify constant dimension codes of length $m+n$, dimension $m$, minimum rank distance $2d$ and size $q^{(m-d+1)n}$ such that each block is disjoint to $S$.

For $d=m$ these subspace codes are vector space partitions, so we get:
\begin{lemma}
	\label{lem:mrd_to_vsp}
	Let $m\leq n$.
	Via the lifting map $\Lambda$, the inner isomorphism classes of $m\times n$ MRD codes of minimum rank distance $m$ correspond to the isomorphism classes of vector space partitions of type $(m^{q^n} n^1)$ of $\F_q^{m+n}$.
	The latter are the same as partial $(m-1)$-spreads of size $m^{q^n}$ in $\PG(m+n-1,q)$ whose hole space is of projective dimension $n-1$.
	Via the map $\Phi$, the inner automorphism group of any of these MRD codes is isomorphic to the automorphism group of the corresponding vector space partition.
\end{lemma}

In the particular case of $(m,n) = (3,4)$, the relevant vector space partitions are given by the partial plane spreads in $\PG(6,2)$ of size $16$ having a hole space of dimension $4$, whose number has been determined in Theorem~\ref{thm:size16}.
As we are in the non-square case $m < n$, any rank metric isomorphism is inner.
We get:

\begin{theorem}
\label{thm:mrd}
\begin{enumerate}[(a)]
	\item
		There are $37$ isomorphism types of vector space partitions of type $(3^{16} 4^1)$ in $\PG(6,2)$.
	\item
		There are $37$ isomorphism types of binary $3\times 4$ MRD codes of minimum rank distance $3$.
\end{enumerate}
\end{theorem}

Given the recent research activity on the isomorphism types of MRD codes, it is worth to give a closer analysis:

\begin{theorem}
	The $37$ classes of binary $3\times 4$ MRD codes of minimum distance $3$ fall into $7$ linear and $30$ non-linear ones.
	The orders of the automorphism groups of the $7$ linear ones are $2688$, $960$, $384$, $288$, $112$, $96$ and $64$.%
	\footnote{By linearity, these automorphisms groups contain the translation subgroup $\{ A\mapsto A + B \mid B\in\mathcal{C}\} \cong (\F_2^4, +)$ of order $16$.}
	Representatives of the $7$ linear MRD codes in descending order of the automorphism group are shown in Table~\ref{tbl:mrd_lin}.
	The orders of the automorphism groups of the $30$ nonlinear ones are given by the distribution $(48^3 42^1 36^1 24^4 20^1 18^1 16^1 12^2 9^1 8^2 6^6 4^2 3^2 2^3)$.%
	Representatives of the $30$ nonlinear MRD codes in descending order of the automorphism group are shown in Tables~\ref{tbl:mrd_nonlin_1},~\ref{tbl:mrd_nonlin_2} and~\ref{tbl:mrd_nonlin_3}.
\end{theorem}

\begin{table}
	\caption{Linear binary $3\times 4$ MRD codes of minimum rank distance $3$}
	\label{tbl:mrd_lin}
	\vspace{-5mm}
    \begin{align*} 
    & \left\langle %2688
    \left(\begin{array}{rrrr}
    0 & 1 & 0 & 0 \\
    0 & 0 & 1 & 0 \\
    0 & 0 & 0 & 1
    \end{array}\right),
    \left(\begin{array}{rrrr}
    1 & 0 & 0 & 0 \\
    0 & 0 & 1 & 1 \\
    0 & 0 & 1 & 0
    \end{array}\right),
    \left(\begin{array}{rrrr}
    0 & 1 & 1 & 0 \\
    1 & 0 & 0 & 1 \\
    0 & 1 & 0 & 0
    \end{array}\right),
    \left(\begin{array}{rrrr}
    0 & 0 & 0 & 1 \\
    1 & 1 & 0 & 1 \\
    1 & 0 & 1 & 0
    \end{array}\right)
    \right\rangle_{\!\F_2} \\
    & \left\langle %960
    \left(\begin{array}{rrrr}
    0 & 1 & 0 & 0 \\
    0 & 0 & 1 & 0 \\
    0 & 0 & 0 & 1
    \end{array}\right),
    \left(\begin{array}{rrrr}
    1 & 0 & 0 & 0 \\
    0 & 0 & 1 & 1 \\
    0 & 0 & 1 & 0
    \end{array}\right),
    \left(\begin{array}{rrrr}
    1 & 0 & 0 & 1 \\
    1 & 0 & 0 & 0 \\
    0 & 1 & 0 & 0
    \end{array}\right),
    \left(\begin{array}{rrrr}
    1 & 1 & 1 & 0 \\
    1 & 1 & 0 & 0 \\
    1 & 0 & 0 & 0
    \end{array}\right)
    \right\rangle_{\!\F_2} \\
    & \left\langle %384
    \left(\begin{array}{rrrr}
    0 & 1 & 0 & 0 \\
    0 & 0 & 1 & 0 \\
    0 & 0 & 0 & 1
    \end{array}\right),
    \left(\begin{array}{rrrr}
    1 & 0 & 0 & 0 \\
    0 & 0 & 1 & 1 \\
    0 & 0 & 1 & 0
    \end{array}\right),
    \left(\begin{array}{rrrr}
    0 & 1 & 1 & 0 \\
    0 & 1 & 0 & 0 \\
    1 & 0 & 0 & 0
    \end{array}\right),
    \left(\begin{array}{rrrr}
    1 & 0 & 1 & 1 \\
    1 & 0 & 0 & 0 \\
    1 & 1 & 0 & 0
    \end{array}\right)
    \right\rangle_{\!\F_2} \\
    & \left\langle %288
    \left(\begin{array}{rrrr}
    0 & 1 & 0 & 0 \\
    0 & 0 & 1 & 0 \\
    0 & 0 & 0 & 1
    \end{array}\right),
    \left(\begin{array}{rrrr}
    1 & 0 & 0 & 0 \\
    0 & 0 & 0 & 1 \\
    0 & 1 & 1 & 0
    \end{array}\right),
    \left(\begin{array}{rrrr}
    1 & 0 & 1 & 0 \\
    1 & 0 & 0 & 0 \\
    1 & 0 & 0 & 1
    \end{array}\right),
    \left(\begin{array}{rrrr}
    0 & 1 & 0 & 1 \\
    0 & 1 & 0 & 0 \\
    1 & 1 & 0 & 0
    \end{array}\right)
    \right\rangle_{\!\F_2} \\
    & \left\langle %112
    \left(\begin{array}{rrrr}
    0 & 1 & 0 & 0 \\
    0 & 0 & 1 & 0 \\
    0 & 0 & 0 & 1
    \end{array}\right),
    \left(\begin{array}{rrrr}
    1 & 0 & 0 & 0 \\
    0 & 0 & 0 & 1 \\
    0 & 1 & 1 & 0
    \end{array}\right),
    \left(\begin{array}{rrrr}
    0 & 0 & 0 & 1 \\
    0 & 1 & 0 & 0 \\
    1 & 0 & 1 & 1
    \end{array}\right),
    \left(\begin{array}{rrrr}
    0 & 1 & 1 & 1 \\
    1 & 0 & 0 & 1 \\
    0 & 1 & 0 & 0
    \end{array}\right)
    \right\rangle_{\!\F_2} \\
    & \left\langle %96
    \left(\begin{array}{rrrr}
    0 & 1 & 0 & 0 \\
    0 & 0 & 1 & 0 \\
    0 & 0 & 0 & 1
    \end{array}\right),
    \left(\begin{array}{rrrr}
    1 & 0 & 0 & 0 \\
    0 & 0 & 0 & 1 \\
    0 & 1 & 1 & 0
    \end{array}\right),
    \left(\begin{array}{rrrr}
    0 & 0 & 1 & 0 \\
    0 & 1 & 1 & 1 \\
    1 & 0 & 0 & 0
    \end{array}\right),
    \left(\begin{array}{rrrr}
    0 & 0 & 1 & 1 \\
    1 & 0 & 0 & 0 \\
    1 & 0 & 1 & 0
    \end{array}\right)
    \right\rangle_{\!\F_2} \\
    & \left\langle %64
    \left(\begin{array}{rrrr}
    0 & 1 & 0 & 0 \\
    0 & 0 & 1 & 0 \\
    0 & 0 & 0 & 1
    \end{array}\right),
    \left(\begin{array}{rrrr}
    1 & 0 & 0 & 0 \\
    0 & 0 & 1 & 1 \\
    0 & 0 & 1 & 0
    \end{array}\right),
    \left(\begin{array}{rrrr}
    0 & 0 & 1 & 0 \\
    1 & 1 & 0 & 0 \\
    0 & 1 & 0 & 0
    \end{array}\right),
    \left(\begin{array}{rrrr}
    1 & 0 & 1 & 1 \\
    0 & 1 & 0 & 0 \\
    1 & 0 & 1 & 0
    \end{array}\right)
    \right\rangle_{\!\F_2}
	\end{align*}
\end{table}

\begin{table}
    \caption{Nonlinear binary $3\times 4$ MRD codes of minimum rank distance $3$, part 1}
    \label{tbl:mrd_nonlin_1}
    \centering
\vspace{2mm}
    $\begin{array}{r@{,}c@{,}c@{,}c@{,}c@{,}c@{,}c@{,}l}
%1
\Big\{\left(\begin{smallmatrix}0&0&0&0\\0&0&0&0\\0&0&0&0\end{smallmatrix}\right)&
\left(\begin{smallmatrix}0&1&0&0\\0&0&1&0\\0&0&0&1\end{smallmatrix}\right)&
\left(\begin{smallmatrix}1&0&0&0\\0&1&1&0\\0&1&0&0\end{smallmatrix}\right)&
\left(\begin{smallmatrix}0&0&1&0\\0&0&0&1\\1&1&0&0\end{smallmatrix}\right)&
\left(\begin{smallmatrix}1&0&1&1\\0&1&0&0\\0&0&1&0\end{smallmatrix}\right)&
\left(\begin{smallmatrix}0&1&0&1\\1&0&0&0\\0&1&1&0\end{smallmatrix}\right)&
\left(\begin{smallmatrix}0&0&0&1\\0&0&1&1\\1&0&1&0\end{smallmatrix}\right)&
\left(\begin{smallmatrix}0&0&1&1\\1&0&1&0\\0&1&0&1\end{smallmatrix}\right),\\[+2mm]
\left(\begin{smallmatrix}1&1&0&0\\1&1&0&1\\1&0&0&0\end{smallmatrix}\right)&
\left(\begin{smallmatrix}1&1&1&0\\1&0&0&1\\0&0&1&1\end{smallmatrix}\right)&
\left(\begin{smallmatrix}1&0&0&1\\0&1&0&1\\1&1&1&0\end{smallmatrix}\right)&
\left(\begin{smallmatrix}0&1&1&0\\1&1&1&0\\1&0&0&1\end{smallmatrix}\right)&
\left(\begin{smallmatrix}0&1&1&1\\1&1&0&0\\1&1&0&1\end{smallmatrix}\right)&
\left(\begin{smallmatrix}1&0&1&0\\1&1&1&1\\1&0&1&1\end{smallmatrix}\right)&
\left(\begin{smallmatrix}1&1&0&1\\0&1&1&1\\1&1&1&1\end{smallmatrix}\right)&
\left(\begin{smallmatrix}1&1&1&1\\1&0&1&1\\0&1&1&1\end{smallmatrix}\right)\Big\}\\[+4mm]
%2
\Big\{\left(\begin{smallmatrix}0&0&0&0\\0&0&0&0\\0&0&0&0\end{smallmatrix}\right)&
\left(\begin{smallmatrix}0&1&0&0\\0&0&1&0\\0&0&0&1\end{smallmatrix}\right)&
\left(\begin{smallmatrix}1&0&0&0\\0&0&0&1\\0&1&1&0\end{smallmatrix}\right)&
\left(\begin{smallmatrix}0&0&1&1\\1&0&0&0\\0&0&1&0\end{smallmatrix}\right)&
\left(\begin{smallmatrix}0&0&1&0\\1&1&0&0\\1&0&1&0\end{smallmatrix}\right)&
\left(\begin{smallmatrix}0&1&0&1\\0&1&0&0\\1&1&0&0\end{smallmatrix}\right)&
\left(\begin{smallmatrix}0&0&0&1\\0&1&1&0\\1&1&0&1\end{smallmatrix}\right)&
\left(\begin{smallmatrix}1&0&1&0\\1&0&0&1\\0&1&0&1\end{smallmatrix}\right),\\[+2mm]
\left(\begin{smallmatrix}1&1&1&1\\0&1&0&1\\1&0&0&0\end{smallmatrix}\right)&
\left(\begin{smallmatrix}0&1&1&1\\1&0&1&0\\0&0&1&1\end{smallmatrix}\right)&
\left(\begin{smallmatrix}1&1&0&0\\0&0&1&1\\0&1&1&1\end{smallmatrix}\right)&
\left(\begin{smallmatrix}1&1&1&0\\1&0&1&1\\0&1&0&0\end{smallmatrix}\right)&
\left(\begin{smallmatrix}0&1&1&0\\1&1&1&0\\1&0&1&1\end{smallmatrix}\right)&
\left(\begin{smallmatrix}1&0&1&1\\0&1&1&1\\1&0&0&1\end{smallmatrix}\right)&
\left(\begin{smallmatrix}1&0&0&1\\1&1&0&1\\1&1&1&1\end{smallmatrix}\right)&
\left(\begin{smallmatrix}1&1&0&1\\1&1&1&1\\1&1&1&0\end{smallmatrix}\right)\Big\}\\[+4mm]
%3
\Big\{\left(\begin{smallmatrix}0&0&0&0\\0&0&0&0\\0&0&0&0\end{smallmatrix}\right)&
\left(\begin{smallmatrix}1&0&0&0\\0&0&1&1\\0&0&1&0\end{smallmatrix}\right)&
\left(\begin{smallmatrix}1&0&1&1\\1&0&0&0\\0&1&0&0\end{smallmatrix}\right)&
\left(\begin{smallmatrix}0&1&0&1\\1&1&0&0\\1&0&0&0\end{smallmatrix}\right)&
\left(\begin{smallmatrix}0&1&1&1\\0&0&1&0\\0&0&0&1\end{smallmatrix}\right)&
\left(\begin{smallmatrix}0&0&1&0\\0&1&0&1\\1&1&0&0\end{smallmatrix}\right)&
\left(\begin{smallmatrix}0&0&0&1\\1&1&1&0\\1&0&0&1\end{smallmatrix}\right)&
\left(\begin{smallmatrix}0&1&0&0\\1&0&0&1\\0&1&1&1\end{smallmatrix}\right),\\[+2mm]
\left(\begin{smallmatrix}1&1&0&0\\1&0&1&0\\0&1&0&1\end{smallmatrix}\right)&
\left(\begin{smallmatrix}1&1&1&1\\0&0&0&1\\0&0&1&1\end{smallmatrix}\right)&
\left(\begin{smallmatrix}0&0&1&1\\1&0&1&1\\0&1&1&0\end{smallmatrix}\right)&
\left(\begin{smallmatrix}1&0&1&0\\0&1&1&0\\1&1&1&0\end{smallmatrix}\right)&
\left(\begin{smallmatrix}1&0&0&1\\1&1&0&1\\1&0&1&1\end{smallmatrix}\right)&
\left(\begin{smallmatrix}0&1&1&0\\0&1&1&1\\1&1&0&1\end{smallmatrix}\right)&
\left(\begin{smallmatrix}1&1&1&0\\0&1&0&0\\1&1&1&1\end{smallmatrix}\right)&
\left(\begin{smallmatrix}1&1&0&1\\1&1&1&1\\1&0&1&0\end{smallmatrix}\right)\Big\}\\[+4mm]
%4
\Big\{\left(\begin{smallmatrix}0&0&0&0\\0&0&0&0\\0&0&0&0\end{smallmatrix}\right)&
\left(\begin{smallmatrix}0&1&0&0\\0&0&1&0\\0&0&0&1\end{smallmatrix}\right)&
\left(\begin{smallmatrix}1&1&0&0\\0&1&0&0\\1&0&1&0\end{smallmatrix}\right)&
\left(\begin{smallmatrix}1&0&0&1\\1&0&1&0\\1&0&0&0\end{smallmatrix}\right)&
\left(\begin{smallmatrix}0&0&1&1\\0&0&0&1\\1&1&0&1\end{smallmatrix}\right)&
\left(\begin{smallmatrix}1&1&1&0\\0&1&0&1\\0&0&1&0\end{smallmatrix}\right)&
\left(\begin{smallmatrix}1&1&0&1\\1&0&0&0\\1&0&0&1\end{smallmatrix}\right)&
\left(\begin{smallmatrix}1&0&0&0\\0&1&1&0\\1&0&1&1\end{smallmatrix}\right),\\[+2mm]
\left(\begin{smallmatrix}0&0&1&0\\1&1&0&1\\0&1&0&1\end{smallmatrix}\right)&
\left(\begin{smallmatrix}0&1&1&1\\0&0&1&1\\1&1&0&0\end{smallmatrix}\right)&
\left(\begin{smallmatrix}1&0&1&0\\0&1&1&1\\0&0&1&1\end{smallmatrix}\right)&
\left(\begin{smallmatrix}1&0&1&1\\1&1&0&0\\0&1&1&0\end{smallmatrix}\right)&
\left(\begin{smallmatrix}0&0&0&1\\1&0&1&1\\1&1&1&0\end{smallmatrix}\right)&
\left(\begin{smallmatrix}0&1&1&0\\1&1&1&1\\0&1&0&0\end{smallmatrix}\right)&
\left(\begin{smallmatrix}0&1&0&1\\1&0&0&1\\1&1&1&1\end{smallmatrix}\right)&
\left(\begin{smallmatrix}1&1&1&1\\1&1&1&0\\0&1&1&1\end{smallmatrix}\right)\Big\}\\[+4mm]
%5
\Big\{\left(\begin{smallmatrix}0&0&0&0\\0&0&0&0\\0&0&0&0\end{smallmatrix}\right)&
\left(\begin{smallmatrix}1&0&1&0\\1&0&0&0\\0&1&0&0\end{smallmatrix}\right)&
\left(\begin{smallmatrix}0&0&0&1\\0&1&0&1\\0&1&1&0\end{smallmatrix}\right)&
\left(\begin{smallmatrix}1&0&0&0\\0&0&1&0\\1&1&0&1\end{smallmatrix}\right)&
\left(\begin{smallmatrix}1&0&0&1\\0&1&1&0\\0&0&0&1\end{smallmatrix}\right)&
\left(\begin{smallmatrix}0&1&1&1\\0&1&0&0\\1&0&1&0\end{smallmatrix}\right)&
\left(\begin{smallmatrix}0&1&0&1\\1&0&0&1\\0&0&1&1\end{smallmatrix}\right)&
\left(\begin{smallmatrix}0&0&1&1\\1&0&1&1\\0&0&1&0\end{smallmatrix}\right),\\[+2mm]
\left(\begin{smallmatrix}0&1&0&0\\0&0&1&1\\1&0&1&1\end{smallmatrix}\right)&
\left(\begin{smallmatrix}1&1&1&1\\0&0&0&1\\0&1&0&1\end{smallmatrix}\right)&
\left(\begin{smallmatrix}1&1&0&1\\1&1&0&0\\1&0&0&1\end{smallmatrix}\right)&
\left(\begin{smallmatrix}1&0&1&1\\1&1&1&0\\1&0&0&0\end{smallmatrix}\right)&
\left(\begin{smallmatrix}1&1&0&0\\1&0&1&0\\0&1&1&1\end{smallmatrix}\right)&
\left(\begin{smallmatrix}1&1&1&0\\0&1&1&1\\1&1&0&0\end{smallmatrix}\right)&
\left(\begin{smallmatrix}0&0&1&0\\1&1&1&1\\1&1&1&0\end{smallmatrix}\right)&
\left(\begin{smallmatrix}0&1&1&0\\1&1&0&1\\1&1&1&1\end{smallmatrix}\right)\Big\}\\[+4mm]
%6
\Big\{\left(\begin{smallmatrix}0&0&0&0\\0&0&0&0\\0&0&0&0\end{smallmatrix}\right)&
\left(\begin{smallmatrix}0&0&1&0\\0&0&0&1\\1&1&0&0\end{smallmatrix}\right)&
\left(\begin{smallmatrix}0&0&0&1\\0&1&1&0\\0&1&0&0\end{smallmatrix}\right)&
\left(\begin{smallmatrix}0&0&1&1\\0&0&1&0\\1&0&0&1\end{smallmatrix}\right)&
\left(\begin{smallmatrix}1&0&1&0\\0&1&0&0\\0&0&1&1\end{smallmatrix}\right)&
\left(\begin{smallmatrix}0&1&0&1\\1&1&0&1\\0&0&0&1\end{smallmatrix}\right)&
\left(\begin{smallmatrix}1&1&1&0\\1&0&1&0\\0&0&1&0\end{smallmatrix}\right)&
\left(\begin{smallmatrix}0&1&0&0\\1&0&1&1\\0&1&0&1\end{smallmatrix}\right),\\[+2mm]
\left(\begin{smallmatrix}1&1&0&1\\1&0&0&0\\1&0&1&1\end{smallmatrix}\right)&
\left(\begin{smallmatrix}1&0&1&1\\1&0&0&1\\0&1&1&0\end{smallmatrix}\right)&
\left(\begin{smallmatrix}0&1&1&0\\1&1&0&0\\1&1&0&1\end{smallmatrix}\right)&
\left(\begin{smallmatrix}1&0&0&0\\0&1&0&1\\1&1&1&1\end{smallmatrix}\right)&
\left(\begin{smallmatrix}0&1&1&1\\1&1&1&0\\1&0&0&0\end{smallmatrix}\right)&
\left(\begin{smallmatrix}1&1&0&0\\0&1&1&1\\1&1&1&0\end{smallmatrix}\right)&
\left(\begin{smallmatrix}1&0&0&1\\1&1&1&1\\1&0&1&0\end{smallmatrix}\right)&
\left(\begin{smallmatrix}1&1&1&1\\0&0&1&1\\0&1&1&1\end{smallmatrix}\right)\Big\}\\[+4mm]
%7
\Big\{\left(\begin{smallmatrix}0&0&0&0\\0&0&0&0\\0&0&0&0\end{smallmatrix}\right)&
\left(\begin{smallmatrix}0&1&0&0\\1&0&0&0\\1&0&0&1\end{smallmatrix}\right)&
\left(\begin{smallmatrix}0&0&1&1\\0&0&1&0\\1&0&0&0\end{smallmatrix}\right)&
\left(\begin{smallmatrix}0&0&0&1\\0&0&1&1\\0&1&0&1\end{smallmatrix}\right)&
\left(\begin{smallmatrix}0&0&1&0\\0&0&0&1\\1&1&0&1\end{smallmatrix}\right)&
\left(\begin{smallmatrix}1&0&1&0\\0&1&0&0\\0&0&1&1\end{smallmatrix}\right)&
\left(\begin{smallmatrix}0&1&1&1\\1&1&0&0\\0&1&0&0\end{smallmatrix}\right)&
\left(\begin{smallmatrix}1&1&1&0\\0&1&1&0\\0&0&1&0\end{smallmatrix}\right),\\[+2mm]
\left(\begin{smallmatrix}1&1&0&0\\0&1&0&1\\1&0&1&1\end{smallmatrix}\right)&
\left(\begin{smallmatrix}0&1&0&1\\1&0&1&0\\0&1&1&1\end{smallmatrix}\right)&
\left(\begin{smallmatrix}1&0&1&1\\1&1&1&0\\0&0&0&1\end{smallmatrix}\right)&
\left(\begin{smallmatrix}1&0&0&0\\1&1&1&1\\0&1&1&0\end{smallmatrix}\right)&
\left(\begin{smallmatrix}1&1&0&1\\0&1&1&1\\1&1&0&0\end{smallmatrix}\right)&
\left(\begin{smallmatrix}1&1&1&1\\1&0&0&1\\1&0&1&0\end{smallmatrix}\right)&
\left(\begin{smallmatrix}0&1&1&0\\1&0&1&1\\1&1&1&0\end{smallmatrix}\right)&
\left(\begin{smallmatrix}1&0&0&1\\1&1&0&1\\1&1&1&1\end{smallmatrix}\right)\Big\}\\[+4mm]
%8
\Big\{\left(\begin{smallmatrix}0&0&0&0\\0&0&0&0\\0&0&0&0\end{smallmatrix}\right)&
\left(\begin{smallmatrix}1&0&0&0\\0&0&1&0\\0&0&0&1\end{smallmatrix}\right)&
\left(\begin{smallmatrix}0&0&1&0\\0&1&0&0\\1&0&0&0\end{smallmatrix}\right)&
\left(\begin{smallmatrix}0&1&0&0\\0&0&0&1\\0&0&1&1\end{smallmatrix}\right)&
\left(\begin{smallmatrix}1&1&1&0\\0&0&1&1\\0&0&1&0\end{smallmatrix}\right)&
\left(\begin{smallmatrix}0&0&1&1\\1&0&1&0\\1&1&0&0\end{smallmatrix}\right)&
\left(\begin{smallmatrix}0&0&0&1\\1&1&1&0\\0&1&0&1\end{smallmatrix}\right)&
\left(\begin{smallmatrix}1&0&0&1\\0&1&1&1\\1&0&1&0\end{smallmatrix}\right),\\[+2mm]
\left(\begin{smallmatrix}1&1&1&1\\1&1&0&0\\0&1&0&0\end{smallmatrix}\right)&
\left(\begin{smallmatrix}0&1&0&1\\1&0&0&1\\1&1&0&1\end{smallmatrix}\right)&
\left(\begin{smallmatrix}1&0&1&0\\1&0&0&0\\1&1&1&1\end{smallmatrix}\right)&
\left(\begin{smallmatrix}0&1&1&1\\0&1&1&0\\1&0&0&1\end{smallmatrix}\right)&
\left(\begin{smallmatrix}1&1&0&1\\0&1&0&1\\1&0&1&1\end{smallmatrix}\right)&
\left(\begin{smallmatrix}1&1&0&0\\1&0&1&1\\1&1&1&0\end{smallmatrix}\right)&
\left(\begin{smallmatrix}0&1&1&0\\1&1&0&1\\0&1&1&1\end{smallmatrix}\right)&
\left(\begin{smallmatrix}1&0&1&1\\1&1&1&1\\0&1&1&0\end{smallmatrix}\right)\Big\}\\[+4mm]
%9
\Big\{\left(\begin{smallmatrix}0&0&0&0\\0&0&0&0\\0&0&0&0\end{smallmatrix}\right)&
\left(\begin{smallmatrix}1&0&0&0\\0&0&1&1\\0&0&1&0\end{smallmatrix}\right)&
\left(\begin{smallmatrix}0&1&0&0\\1&0&0&1\\0&1&0&1\end{smallmatrix}\right)&
\left(\begin{smallmatrix}0&0&0&1\\1&0&1&1\\0&1&0&0\end{smallmatrix}\right)&
\left(\begin{smallmatrix}1&0&0&1\\1&0&0&0\\0&1&1&0\end{smallmatrix}\right)&
\left(\begin{smallmatrix}0&1&1&1\\0&0&1&0\\0&0&0&1\end{smallmatrix}\right)&
\left(\begin{smallmatrix}0&1&1&0\\0&1&0&0\\1&1&1&0\end{smallmatrix}\right)&
\left(\begin{smallmatrix}1&0&1&1\\1&1&0&0\\1&0&0&0\end{smallmatrix}\right),\\[+2mm]
\left(\begin{smallmatrix}1&1&1&1\\0&0&0&1\\0&0&1&1\end{smallmatrix}\right)&
\left(\begin{smallmatrix}1&0&1&0\\0&1&0&1\\1&1&0&1\end{smallmatrix}\right)&
\left(\begin{smallmatrix}1&1&0&0\\1&0&1&0\\0&1&1&1\end{smallmatrix}\right)&
\left(\begin{smallmatrix}0&0&1&0\\0&1&1&0\\1&1&1&1\end{smallmatrix}\right)&
\left(\begin{smallmatrix}1&1&0&1\\1&1&1&0\\1&0&0&1\end{smallmatrix}\right)&
\left(\begin{smallmatrix}0&0&1&1\\1&1&1&1\\1&0&1&0\end{smallmatrix}\right)&
\left(\begin{smallmatrix}1&1&1&0\\0&1&1&1\\1&1&0&0\end{smallmatrix}\right)&
\left(\begin{smallmatrix}0&1&0&1\\1&1&0&1\\1&0&1&1\end{smallmatrix}\right)\Big\}\\[+4mm]
%10
\Big\{\left(\begin{smallmatrix}0&0&0&0\\0&0&0&0\\0&0&0&0\end{smallmatrix}\right)&
\left(\begin{smallmatrix}0&1&0&0\\0&1&1&1\\0&0&0&1\end{smallmatrix}\right)&
\left(\begin{smallmatrix}0&0&0&1\\1&1&0&0\\0&1&0&1\end{smallmatrix}\right)&
\left(\begin{smallmatrix}1&0&1&0\\0&1&0&0\\0&0&1&1\end{smallmatrix}\right)&
\left(\begin{smallmatrix}0&0&1&1\\0&0&0&1\\1&0&0&1\end{smallmatrix}\right)&
\left(\begin{smallmatrix}1&0&0&0\\0&0&1&0\\1&1&1&1\end{smallmatrix}\right)&
\left(\begin{smallmatrix}0&1&1&1\\0&0&1&1\\1&0&0&0\end{smallmatrix}\right)&
\left(\begin{smallmatrix}0&0&1&0\\1&1&0&1\\1&1&0&0\end{smallmatrix}\right),\\[+2mm]
\left(\begin{smallmatrix}1&1&1&0\\0&1&1&0\\0&0&1&0\end{smallmatrix}\right)&
\left(\begin{smallmatrix}1&0&0&1\\0&1&0&1\\1&0&1&0\end{smallmatrix}\right)&
\left(\begin{smallmatrix}0&1&0&1\\1&1&1&0\\0&1&0&0\end{smallmatrix}\right)&
\left(\begin{smallmatrix}1&0&1&1\\1&0&1&0\\0&1&1&0\end{smallmatrix}\right)&
\left(\begin{smallmatrix}1&1&0&1\\1&0&0&0\\1&0&1&1\end{smallmatrix}\right)&
\left(\begin{smallmatrix}1&1&0&0\\1&0&1&1\\1&1&1&0\end{smallmatrix}\right)&
\left(\begin{smallmatrix}0&1&1&0\\1&1&1&1\\1&1&0&1\end{smallmatrix}\right)&
\left(\begin{smallmatrix}1&1&1&1\\1&0&0&1\\0&1&1&1\end{smallmatrix}\right)\Big\}\\[+4mm]
    \end{array}$
\end{table}

\begin{table}
    \caption{Nonlinear binary $3\times 4$ MRD codes of minimum rank distance $3$, part 2}
    \label{tbl:mrd_nonlin_2}
    \centering
\vspace{2mm}
    $\begin{array}{r@{,}c@{,}c@{,}c@{,}c@{,}c@{,}c@{,}l}
%11
\Big\{\left(\begin{smallmatrix}0&0&0&0\\0&0&0&0\\0&0&0&0\end{smallmatrix}\right)&
\left(\begin{smallmatrix}1&1&0&0\\0&0&0&1\\1&0&0&0\end{smallmatrix}\right)&
\left(\begin{smallmatrix}1&0&1&0\\0&1&0&0\\0&1&1&0\end{smallmatrix}\right)&
\left(\begin{smallmatrix}0&0&0&1\\1&0&1&1\\0&1&0&0\end{smallmatrix}\right)&
\left(\begin{smallmatrix}0&0&1&1\\0&1&1&0\\0&0&0&1\end{smallmatrix}\right)&
\left(\begin{smallmatrix}1&0&0&0\\0&0&1&1\\1&0&0&1\end{smallmatrix}\right)&
\left(\begin{smallmatrix}0&1&1&1\\1&0&0&0\\1&1&0&0\end{smallmatrix}\right)&
\left(\begin{smallmatrix}1&1&1&0\\0&0&1&0\\0&0&1&1\end{smallmatrix}\right),\\[+2mm]
\left(\begin{smallmatrix}0&0&1&0\\0&1&0&1\\1&1&0&1\end{smallmatrix}\right)&
\left(\begin{smallmatrix}1&1&0&1\\1&1&0&0\\1&0&1&0\end{smallmatrix}\right)&
\left(\begin{smallmatrix}0&1&1&0\\1&1&1&0\\0&1&0&1\end{smallmatrix}\right)&
\left(\begin{smallmatrix}0&1&0&0\\1&1&0&1\\1&1&1&0\end{smallmatrix}\right)&
\left(\begin{smallmatrix}1&0&1&1\\1&1&1&1\\0&0&1&0\end{smallmatrix}\right)&
\left(\begin{smallmatrix}1&0&0&1\\1&0&1&0\\1&1&1&1\end{smallmatrix}\right)&
\left(\begin{smallmatrix}0&1&0&1\\0&1&1&1\\1&0&1&1\end{smallmatrix}\right)&
\left(\begin{smallmatrix}1&1&1&1\\1&0&0&1\\0&1&1&1\end{smallmatrix}\right)\Big\}\\[+4mm]
%12
\Big\{\left(\begin{smallmatrix}0&0&0&0\\0&0&0&0\\0&0&0&0\end{smallmatrix}\right)&
\left(\begin{smallmatrix}1&1&1&0\\1&0&0&0\\0&0&0&1\end{smallmatrix}\right)&
\left(\begin{smallmatrix}0&1&1&0\\1&1&0&0\\0&1&0&0\end{smallmatrix}\right)&
\left(\begin{smallmatrix}1&0&0&1\\0&0&1&0\\0&0&1&1\end{smallmatrix}\right)&
\left(\begin{smallmatrix}0&1&0&0\\1&0&0&1\\1&1&0&0\end{smallmatrix}\right)&
\left(\begin{smallmatrix}0&1&0&1\\0&0&0&1\\1&1&1&0\end{smallmatrix}\right)&
\left(\begin{smallmatrix}0&1&1&1\\0&1&0&0\\0&1&1&0\end{smallmatrix}\right)&
\left(\begin{smallmatrix}0&0&1&0\\1&1&1&0\\0&1&0&1\end{smallmatrix}\right),\\[+2mm]
\left(\begin{smallmatrix}1&1&0&0\\1&0&1&1\\1&0&0&0\end{smallmatrix}\right)&
\left(\begin{smallmatrix}0&0&0&1\\0&0&1&1\\1&1&1&1\end{smallmatrix}\right)&
\left(\begin{smallmatrix}0&0&1&1\\1&1&0&1\\1&0&1&0\end{smallmatrix}\right)&
\left(\begin{smallmatrix}1&1&1&1\\0&1&1&0\\0&0&1&0\end{smallmatrix}\right)&
\left(\begin{smallmatrix}1&0&0&0\\0&1&1&1\\1&1&0&1\end{smallmatrix}\right)&
\left(\begin{smallmatrix}1&0&1&0\\1&1&1&1\\1&0&0&1\end{smallmatrix}\right)&
\left(\begin{smallmatrix}1&1&0&1\\0&1&0&1\\1&0&1&1\end{smallmatrix}\right)&
\left(\begin{smallmatrix}1&0&1&1\\1&0&1&0\\0&1&1&1\end{smallmatrix}\right)\Big\}\\[+4mm]
%13
\Big\{\left(\begin{smallmatrix}0&0&0&0\\0&0&0&0\\0&0&0&0\end{smallmatrix}\right)&
\left(\begin{smallmatrix}1&0&0&1\\0&1&0&0\\0&0&1&0\end{smallmatrix}\right)&
\left(\begin{smallmatrix}0&0&1&0\\1&0&0&0\\0&0&1&1\end{smallmatrix}\right)&
\left(\begin{smallmatrix}1&0&0&0\\1&1&1&0\\0&0&0&1\end{smallmatrix}\right)&
\left(\begin{smallmatrix}0&1&1&0\\1&0&1&0\\0&1&0&0\end{smallmatrix}\right)&
\left(\begin{smallmatrix}0&1&0&1\\0&0&0&1\\1&0&0&1\end{smallmatrix}\right)&
\left(\begin{smallmatrix}0&0&0&1\\0&0&1&0\\1&1&1&1\end{smallmatrix}\right)&
\left(\begin{smallmatrix}0&1&0&0\\0&1&0&1\\0&1&1&1\end{smallmatrix}\right),\\[+2mm]
\left(\begin{smallmatrix}1&1&0&0\\0&0&1&1\\1&0&1&0\end{smallmatrix}\right)&
\left(\begin{smallmatrix}1&0&1&1\\1&1&0&0\\0&1&1&0\end{smallmatrix}\right)&
\left(\begin{smallmatrix}0&0&1&1\\1&0&1&1\\1&1&0&0\end{smallmatrix}\right)&
\left(\begin{smallmatrix}1&1&0&1\\0&1&1&0\\0&1&0&1\end{smallmatrix}\right)&
\left(\begin{smallmatrix}0&1&1&1\\1&0&0&1\\1&1&0&1\end{smallmatrix}\right)&
\left(\begin{smallmatrix}1&1&1&1\\1&1&0&1\\1&0&0&0\end{smallmatrix}\right)&
\left(\begin{smallmatrix}1&1&1&0\\0&1&1&1\\1&0&1&1\end{smallmatrix}\right)&
\left(\begin{smallmatrix}1&0&1&0\\1&1&1&1\\1&1&1&0\end{smallmatrix}\right)\Big\}\\[+4mm]
%14
\Big\{\left(\begin{smallmatrix}0&0&0&0\\0&0&0&0\\0&0&0&0\end{smallmatrix}\right)&
\left(\begin{smallmatrix}0&0&1&0\\0&0&0&1\\1&0&0&0\end{smallmatrix}\right)&
\left(\begin{smallmatrix}0&1&0&0\\0&0&1&0\\0&0&0&1\end{smallmatrix}\right)&
\left(\begin{smallmatrix}1&0&0&1\\0&1&0&0\\0&0&1&0\end{smallmatrix}\right)&
\left(\begin{smallmatrix}1&1&0&0\\0&1&1&0\\0&1&0&0\end{smallmatrix}\right)&
\left(\begin{smallmatrix}1&0&0&0\\1&0&0&1\\0&1&1&1\end{smallmatrix}\right)&
\left(\begin{smallmatrix}0&1&1&0\\1&0&1&0\\0&1&0&1\end{smallmatrix}\right)&
\left(\begin{smallmatrix}1&1&0&1\\1&0&0&0\\1&0&0&1\end{smallmatrix}\right),\\[+2mm]
\left(\begin{smallmatrix}0&1&0&1\\1&1&1&0\\0&0&1&1\end{smallmatrix}\right)&
\left(\begin{smallmatrix}0&0&1&1\\1&1&0&1\\1&0&1&0\end{smallmatrix}\right)&
\left(\begin{smallmatrix}1&0&1&0\\0&1&0&1\\1&1&0&1\end{smallmatrix}\right)&
\left(\begin{smallmatrix}0&0&0&1\\0&0&1&1\\1&1&1&1\end{smallmatrix}\right)&
\left(\begin{smallmatrix}1&0&1&1\\1&1&0&0\\0&1&1&0\end{smallmatrix}\right)&
\left(\begin{smallmatrix}1&1&1&0\\0&1&1&1\\1&1&0&0\end{smallmatrix}\right)&
\left(\begin{smallmatrix}1&1&1&1\\1&0&1&1\\1&1&1&0\end{smallmatrix}\right)&
\left(\begin{smallmatrix}0&1&1&1\\1&1&1&1\\1&0&1&1\end{smallmatrix}\right)\Big\}\\[+4mm]
%15
\Big\{\left(\begin{smallmatrix}0&0&0&0\\0&0&0&0\\0&0&0&0\end{smallmatrix}\right)&
\left(\begin{smallmatrix}0&0&0&1\\1&0&0&0\\0&1&0&1\end{smallmatrix}\right)&
\left(\begin{smallmatrix}0&1&0&0\\0&1&1&0\\1&0&1&0\end{smallmatrix}\right)&
\left(\begin{smallmatrix}1&0&0&1\\0&0&1&0\\0&1&1&0\end{smallmatrix}\right)&
\left(\begin{smallmatrix}0&1&1&1\\0&1&0&0\\0&0&0&1\end{smallmatrix}\right)&
\left(\begin{smallmatrix}0&0&1&1\\1&1&0&0\\0&1&0&0\end{smallmatrix}\right)&
\left(\begin{smallmatrix}1&0&1&0\\1&0&1&1\\0&0&1&0\end{smallmatrix}\right)&
\left(\begin{smallmatrix}0&0&1&0\\1&1&1&1\\1&0&0&0\end{smallmatrix}\right),\\[+2mm]
\left(\begin{smallmatrix}1&1&0&0\\1&0&1&0\\1&0&0&1\end{smallmatrix}\right)&
\left(\begin{smallmatrix}0&1&0&1\\0&0&0&1\\1&1&1&0\end{smallmatrix}\right)&
\left(\begin{smallmatrix}1&1&1&0\\1&0&0&1\\0&0&1&1\end{smallmatrix}\right)&
\left(\begin{smallmatrix}1&0&0&0\\1&1&0&1\\1&0&1&1\end{smallmatrix}\right)&
\left(\begin{smallmatrix}1&0&1&1\\0&1&0&1\\0&1&1&1\end{smallmatrix}\right)&
\left(\begin{smallmatrix}1&1&0&1\\0&1&1&1\\1&1&0&0\end{smallmatrix}\right)&
\left(\begin{smallmatrix}0&1&1&0\\1&1&1&0\\1&1&1&1\end{smallmatrix}\right)&
\left(\begin{smallmatrix}1&1&1&1\\0&0&1&1\\1&1&0&1\end{smallmatrix}\right)\Big\}\\[+4mm]
%16
\Big\{\left(\begin{smallmatrix}0&0&0&0\\0&0&0&0\\0&0&0&0\end{smallmatrix}\right)&
\left(\begin{smallmatrix}1&0&0&0\\0&1&0&0\\0&0&1&1\end{smallmatrix}\right)&
\left(\begin{smallmatrix}0&0&1&1\\1&0&0&0\\0&0&0&1\end{smallmatrix}\right)&
\left(\begin{smallmatrix}0&0&0&1\\1&1&0&0\\1&0&0&0\end{smallmatrix}\right)&
\left(\begin{smallmatrix}0&1&0&1\\1&0&1&0\\0&0&1&0\end{smallmatrix}\right)&
\left(\begin{smallmatrix}0&0&1&0\\0&1&1&0\\1&1&0&0\end{smallmatrix}\right)&
\left(\begin{smallmatrix}0&1&0&0\\0&0&0&1\\0&1&1&1\end{smallmatrix}\right)&
\left(\begin{smallmatrix}1&1&0&0\\0&0&1&0\\1&0&0&1\end{smallmatrix}\right),\\[+2mm]
\left(\begin{smallmatrix}1&1&1&0\\1&0&0&1\\0&1&0&0\end{smallmatrix}\right)&
\left(\begin{smallmatrix}0&1&1&0\\0&1&0&1\\1&0&1&0\end{smallmatrix}\right)&
\left(\begin{smallmatrix}1&0&1&1\\0&0&1&1\\0&1&1&0\end{smallmatrix}\right)&
\left(\begin{smallmatrix}1&0&1&0\\1&0&1&1\\0&1&0&1\end{smallmatrix}\right)&
\left(\begin{smallmatrix}1&0&0&1\\1&1&1&1\\1&1&0&1\end{smallmatrix}\right)&
\left(\begin{smallmatrix}1&1&0&1\\0&1&1&1\\1&0&1&1\end{smallmatrix}\right)&
\left(\begin{smallmatrix}1&1&1&1\\1&1&0&1\\1&1&1&0\end{smallmatrix}\right)&
\left(\begin{smallmatrix}0&1&1&1\\1&1&1&0\\1&1&1&1\end{smallmatrix}\right)\Big\}\\[+4mm]
%17
\Big\{\left(\begin{smallmatrix}0&0&0&0\\0&0&0&0\\0&0&0&0\end{smallmatrix}\right)&
\left(\begin{smallmatrix}0&1&0&0\\0&0&1&0\\0&0&0&1\end{smallmatrix}\right)&
\left(\begin{smallmatrix}0&1&0&1\\0&1&1&0\\1&0&0&0\end{smallmatrix}\right)&
\left(\begin{smallmatrix}1&0&1&0\\1&0&0&0\\0&0&1&1\end{smallmatrix}\right)&
\left(\begin{smallmatrix}0&0&1&1\\0&1&0&0\\1&0&0&1\end{smallmatrix}\right)&
\left(\begin{smallmatrix}0&0&1&0\\1&1&0&1\\1&1&0&0\end{smallmatrix}\right)&
\left(\begin{smallmatrix}1&1&1&0\\1&0&1&0\\0&0&1&0\end{smallmatrix}\right)&
\left(\begin{smallmatrix}1&0&1&1\\0&0&1&1\\0&1&0&0\end{smallmatrix}\right),\\[+2mm]
\left(\begin{smallmatrix}1&0&0&1\\1&1&0&0\\1&0&1&0\end{smallmatrix}\right)&
\left(\begin{smallmatrix}1&1&0&1\\0&0&0&1\\0&1&0&1\end{smallmatrix}\right)&
\left(\begin{smallmatrix}1&0&0&0\\0&1&0&1\\1&1&1&1\end{smallmatrix}\right)&
\left(\begin{smallmatrix}0&0&0&1\\1&0&1&1\\0&1&1&1\end{smallmatrix}\right)&
\left(\begin{smallmatrix}0&1&1&1\\1&0&0&1\\0&1&1&0\end{smallmatrix}\right)&
\left(\begin{smallmatrix}1&1&0&0\\0&1&1&1\\1&1&1&0\end{smallmatrix}\right)&
\left(\begin{smallmatrix}0&1&1&0\\1&1&1&1\\1&1&0&1\end{smallmatrix}\right)&
\left(\begin{smallmatrix}1&1&1&1\\1&1&1&0\\1&0&1&1\end{smallmatrix}\right)\Big\}\\[+4mm]
%18
\Big\{\left(\begin{smallmatrix}0&0&0&0\\0&0&0&0\\0&0&0&0\end{smallmatrix}\right)&
\left(\begin{smallmatrix}0&1&0&0\\1&0&0&0\\0&0&0&1\end{smallmatrix}\right)&
\left(\begin{smallmatrix}1&0&0&0\\0&1&1&0\\0&0&1&0\end{smallmatrix}\right)&
\left(\begin{smallmatrix}0&0&0&1\\0&1&0&1\\1&0&1&0\end{smallmatrix}\right)&
\left(\begin{smallmatrix}0&0&1&0\\0&0&1&1\\0&1&0&1\end{smallmatrix}\right)&
\left(\begin{smallmatrix}1&1&0&0\\0&1&0&0\\0&0&1&1\end{smallmatrix}\right)&
\left(\begin{smallmatrix}1&0&1&1\\0&0&1&0\\1&0&0&0\end{smallmatrix}\right)&
\left(\begin{smallmatrix}0&0&1&1\\0&0&0&1\\1&1&0&1\end{smallmatrix}\right),\\[+2mm]
\left(\begin{smallmatrix}0&1&1&0\\1&0&0&1\\0&1&1&1\end{smallmatrix}\right)&
\left(\begin{smallmatrix}1&0&0&1\\1&1&1&1\\0&1&0&0\end{smallmatrix}\right)&
\left(\begin{smallmatrix}0&1&0&1\\1&1&1&0\\1&1&0&0\end{smallmatrix}\right)&
\left(\begin{smallmatrix}1&1&1&1\\1&1&0&0\\1&0&0&1\end{smallmatrix}\right)&
\left(\begin{smallmatrix}1&1&1&0\\1&0&1&1\\0&1&1&0\end{smallmatrix}\right)&
\left(\begin{smallmatrix}0&1&1&1\\1&0&1&0\\1&0&1&1\end{smallmatrix}\right)&
\left(\begin{smallmatrix}1&0&1&0\\1&1&0&1\\1&1&1&0\end{smallmatrix}\right)&
\left(\begin{smallmatrix}1&1&0&1\\0&1&1&1\\1&1&1&1\end{smallmatrix}\right)\Big\}\\[+4mm]
%19
\Big\{\left(\begin{smallmatrix}0&0&0&0\\0&0&0&0\\0&0&0&0\end{smallmatrix}\right)&
\left(\begin{smallmatrix}0&1&0&0\\0&0&1&0\\0&0&0&1\end{smallmatrix}\right)&
\left(\begin{smallmatrix}1&0&0&0\\0&0&0&1\\0&1&1&1\end{smallmatrix}\right)&
\left(\begin{smallmatrix}0&0&1&0\\0&1&1&0\\1&0&0&1\end{smallmatrix}\right)&
\left(\begin{smallmatrix}1&0&0&1\\0&1&0&1\\1&0&0&0\end{smallmatrix}\right)&
\left(\begin{smallmatrix}0&1&1&0\\0&1&0&0\\1&1&1&0\end{smallmatrix}\right)&
\left(\begin{smallmatrix}0&1&0&1\\1&1&1&0\\0&0&1&0\end{smallmatrix}\right)&
\left(\begin{smallmatrix}0&0&0&1\\1&0&0&1\\1&0&1&1\end{smallmatrix}\right),\\[+2mm]
\left(\begin{smallmatrix}1&0&1&1\\1&0&0&0\\0&1&1&0\end{smallmatrix}\right)&
\left(\begin{smallmatrix}1&1&0&1\\0&0&1&1\\0&1&0&0\end{smallmatrix}\right)&
\left(\begin{smallmatrix}1&0&1&0\\0&1&1&1\\1&1&0&0\end{smallmatrix}\right)&
\left(\begin{smallmatrix}0&0&1&1\\1&1&0&0\\1&1&0&1\end{smallmatrix}\right)&
\left(\begin{smallmatrix}0&1&1&1\\1&1&0&1\\0&1&0&1\end{smallmatrix}\right)&
\left(\begin{smallmatrix}1&1&1&1\\1&0&1&0\\0&0&1&1\end{smallmatrix}\right)&
\left(\begin{smallmatrix}1&1&1&0\\1&1&1&1\\1&0&1&0\end{smallmatrix}\right)&
\left(\begin{smallmatrix}1&1&0&0\\1&0&1&1\\1&1&1&1\end{smallmatrix}\right)\Big\}\\[+4mm]
%20
\Big\{\left(\begin{smallmatrix}0&0&0&0\\0&0&0&0\\0&0&0&0\end{smallmatrix}\right)&
\left(\begin{smallmatrix}1&0&0&0\\0&0&1&0\\0&0&1&1\end{smallmatrix}\right)&
\left(\begin{smallmatrix}0&0&1&1\\0&1&0&0\\0&0&1&0\end{smallmatrix}\right)&
\left(\begin{smallmatrix}0&1&0&0\\1&1&0&0\\0&1&1&0\end{smallmatrix}\right)&
\left(\begin{smallmatrix}0&1&0&1\\1&0&1&0\\1&0&0&0\end{smallmatrix}\right)&
\left(\begin{smallmatrix}1&1&1&0\\0&1&0&1\\0&0&0&1\end{smallmatrix}\right)&
\left(\begin{smallmatrix}1&1&0&0\\1&0&0&1\\1&0&1&0\end{smallmatrix}\right)&
\left(\begin{smallmatrix}0&0&1&0\\1&1&1&0\\0&1&1&1\end{smallmatrix}\right),\\[+2mm]
\left(\begin{smallmatrix}0&0&0&1\\0&1&1&1\\1&0&1&1\end{smallmatrix}\right)&
\left(\begin{smallmatrix}1&0&1&0\\0&0&1&1\\1&1&0&1\end{smallmatrix}\right)&
\left(\begin{smallmatrix}0&1&1&1\\0&1&1&0\\1&0&0&1\end{smallmatrix}\right)&
\left(\begin{smallmatrix}1&1&1&1\\0&0&0&1\\1&1&0&0\end{smallmatrix}\right)&
\left(\begin{smallmatrix}1&0&0&1\\1&1&0&1\\0&1&0&1\end{smallmatrix}\right)&
\left(\begin{smallmatrix}0&1&1&0\\1&0&1&1\\1&1&1&0\end{smallmatrix}\right)&
\left(\begin{smallmatrix}1&1&0&1\\1&1&1&1\\0&1&0&0\end{smallmatrix}\right)&
\left(\begin{smallmatrix}1&0&1&1\\1&0&0&0\\1&1&1&1\end{smallmatrix}\right)\Big\}\\[+4mm]
    \end{array}$
\end{table}

\begin{table}
    \caption{Nonlinear binary $3\times 4$ MRD codes of minimum rank distance $3$, part 3}
    \label{tbl:mrd_nonlin_3}
\vspace{2mm}
    \centering
    $\begin{array}{r@{,}c@{,}c@{,}c@{,}c@{,}c@{,}c@{,}l}
%21
\Big\{\left(\begin{smallmatrix}0&0&0&0\\0&0&0&0\\0&0&0&0\end{smallmatrix}\right)&
\left(\begin{smallmatrix}1&1&0&1\\0&0&0&1\\0&0&1&0\end{smallmatrix}\right)&
\left(\begin{smallmatrix}0&0&1&0\\0&1&0&1\\1&0&1&0\end{smallmatrix}\right)&
\left(\begin{smallmatrix}1&0&0&1\\0&1&0&0\\1&1&0&0\end{smallmatrix}\right)&
\left(\begin{smallmatrix}1&0&1&0\\1&0&0&0\\1&0&0&1\end{smallmatrix}\right)&
\left(\begin{smallmatrix}0&0&0&1\\0&1&1&1\\0&0&1&1\end{smallmatrix}\right)&
\left(\begin{smallmatrix}1&1&0&0\\0&0&1&0\\0&1&1&1\end{smallmatrix}\right)&
\left(\begin{smallmatrix}1&1&1&0\\1&0&1&0\\1&0&0&0\end{smallmatrix}\right),\\[+2mm]
\left(\begin{smallmatrix}0&1&0&1\\1&1&1&0\\0&1&0&0\end{smallmatrix}\right)&
\left(\begin{smallmatrix}1&0&0&0\\1&1&0&1\\0&1&1&0\end{smallmatrix}\right)&
\left(\begin{smallmatrix}0&1&1&0\\1&1&1&1\\0&0&0&1\end{smallmatrix}\right)&
\left(\begin{smallmatrix}0&0&1&1\\1&0&0&1\\1&1&1&0\end{smallmatrix}\right)&
\left(\begin{smallmatrix}0&1&0&0\\1&0&1&1\\1&1&0&1\end{smallmatrix}\right)&
\left(\begin{smallmatrix}1&1&1&1\\1&1&0&0\\0&1&0&1\end{smallmatrix}\right)&
\left(\begin{smallmatrix}0&1&1&1\\0&1&1&0\\1&0&1&1\end{smallmatrix}\right)&
\left(\begin{smallmatrix}1&0&1&1\\0&0&1&1\\1&1&1&1\end{smallmatrix}\right)\Big\}\\[+4mm]
%22
\Big\{\left(\begin{smallmatrix}0&0&0&0\\0&0&0&0\\0&0&0&0\end{smallmatrix}\right)&
\left(\begin{smallmatrix}0&0&0&1\\1&0&0&0\\0&1&1&0\end{smallmatrix}\right)&
\left(\begin{smallmatrix}1&0&0&0\\0&0&1&1\\0&0&0&1\end{smallmatrix}\right)&
\left(\begin{smallmatrix}0&1&0&1\\0&0&1&0\\0&0&1&1\end{smallmatrix}\right)&
\left(\begin{smallmatrix}0&1&1&1\\0&1&0&0\\1&0&0&0\end{smallmatrix}\right)&
\left(\begin{smallmatrix}1&1&1&1\\0&0&0&1\\0&0&1&0\end{smallmatrix}\right)&
\left(\begin{smallmatrix}0&1&0&0\\1&0&0&1\\1&1&1&0\end{smallmatrix}\right)&
\left(\begin{smallmatrix}1&0&1&0\\1&0&1&1\\0&1&0&0\end{smallmatrix}\right),\\[+2mm]
\left(\begin{smallmatrix}1&0&0&1\\0&1&0&1\\1&1&0&1\end{smallmatrix}\right)&
\left(\begin{smallmatrix}1&1&1&0\\1&1&0&0\\0&1&0&1\end{smallmatrix}\right)&
\left(\begin{smallmatrix}0&0&1&0\\1&1&0&1\\0&1&1&1\end{smallmatrix}\right)&
\left(\begin{smallmatrix}0&1&1&0\\1&1&1&0\\1&0&1&0\end{smallmatrix}\right)&
\left(\begin{smallmatrix}0&0&1&1\\0&1&1&0\\1&0&1&1\end{smallmatrix}\right)&
\left(\begin{smallmatrix}1&1&0&0\\1&0&1&0\\1&1&1&1\end{smallmatrix}\right)&
\left(\begin{smallmatrix}1&1&0&1\\0&1&1&1\\1&1&0&0\end{smallmatrix}\right)&
\left(\begin{smallmatrix}1&0&1&1\\1&1&1&1\\1&0&0&1\end{smallmatrix}\right)\Big\}\\[+4mm]
%23
\Big\{\left(\begin{smallmatrix}0&0&0&0\\0&0&0&0\\0&0&0&0\end{smallmatrix}\right)&
\left(\begin{smallmatrix}0&1&0&0\\0&0&1&0\\0&0&0&1\end{smallmatrix}\right)&
\left(\begin{smallmatrix}1&0&0&0\\0&1&1&0\\0&1&0&0\end{smallmatrix}\right)&
\left(\begin{smallmatrix}0&0&1&0\\1&0&0&0\\0&1&0&1\end{smallmatrix}\right)&
\left(\begin{smallmatrix}1&0&1&1\\0&1&0&0\\0&0&1&0\end{smallmatrix}\right)&
\left(\begin{smallmatrix}0&0&1&1\\0&1&1&1\\1&0&0&0\end{smallmatrix}\right)&
\left(\begin{smallmatrix}0&0&0&1\\1&0&1&0\\1&1&1&0\end{smallmatrix}\right)&
\left(\begin{smallmatrix}1&0&1&0\\0&0&1&1\\1&0&1&1\end{smallmatrix}\right),\\[+2mm]
\left(\begin{smallmatrix}0&1&1&0\\1&0&1&1\\1&1&0&0\end{smallmatrix}\right)&
\left(\begin{smallmatrix}1&1&1&0\\1&0&0&1\\0&0&1&1\end{smallmatrix}\right)&
\left(\begin{smallmatrix}0&1&0&1\\1&1&0&1\\0&1&1&0\end{smallmatrix}\right)&
\left(\begin{smallmatrix}0&1&1&1\\0&1&0&1\\1&0&0&1\end{smallmatrix}\right)&
\left(\begin{smallmatrix}1&0&0&1\\1&1&0&0\\1&1&0&1\end{smallmatrix}\right)&
\left(\begin{smallmatrix}1&1&0&0\\1&1&1&1\\1&0&1&0\end{smallmatrix}\right)&
\left(\begin{smallmatrix}1&1&0&1\\0&0&0&1\\1&1&1&1\end{smallmatrix}\right)&
\left(\begin{smallmatrix}1&1&1&1\\1&1&1&0\\0&1&1&1\end{smallmatrix}\right)\Big\}\\[+4mm]
%24
\Big\{\left(\begin{smallmatrix}0&0&0&0\\0&0&0&0\\0&0&0&0\end{smallmatrix}\right)&
\left(\begin{smallmatrix}1&0&0&1\\0&0&1&0\\0&0&0&1\end{smallmatrix}\right)&
\left(\begin{smallmatrix}0&0&1&0\\0&1&0&1\\1&0&0&1\end{smallmatrix}\right)&
\left(\begin{smallmatrix}0&1&0&1\\0&0&0&1\\0&0&1&1\end{smallmatrix}\right)&
\left(\begin{smallmatrix}0&0&0&1\\1&0&0&1\\1&1&0&0\end{smallmatrix}\right)&
\left(\begin{smallmatrix}1&1&1&0\\0&0&1&1\\0&0&1&0\end{smallmatrix}\right)&
\left(\begin{smallmatrix}0&1&1&0\\0&1&1&1\\1&0&0&0\end{smallmatrix}\right)&
\left(\begin{smallmatrix}0&1&0&0\\1&1&1&0\\0&1&1&0\end{smallmatrix}\right),\\[+2mm]
\left(\begin{smallmatrix}1&1&0&1\\0&1&0&0\\1&0&1&0\end{smallmatrix}\right)&
\left(\begin{smallmatrix}1&0&0&0\\0&1&1&0\\1&0&1&1\end{smallmatrix}\right)&
\left(\begin{smallmatrix}0&0&1&1\\1&1&0&0\\0&1&1&1\end{smallmatrix}\right)&
\left(\begin{smallmatrix}1&0&1&1\\1&1&0&1\\0&1&0&0\end{smallmatrix}\right)&
\left(\begin{smallmatrix}1&0&1&0\\1&0&0&0\\1&1&1&1\end{smallmatrix}\right)&
\left(\begin{smallmatrix}1&1&0&0\\1&1&1&1\\0&1&0&1\end{smallmatrix}\right)&
\left(\begin{smallmatrix}0&1&1&1\\1&0&1&1\\1&1&0&1\end{smallmatrix}\right)&
\left(\begin{smallmatrix}1&1&1&1\\1&0&1&0\\1&1&1&0\end{smallmatrix}\right)\Big\}\\[+4mm]
%25
\Big\{\left(\begin{smallmatrix}0&0&0&0\\0&0&0&0\\0&0&0&0\end{smallmatrix}\right)&
\left(\begin{smallmatrix}0&1&0&1\\1&0&0&1\\0&1&0&0\end{smallmatrix}\right)&
\left(\begin{smallmatrix}1&0&1&1\\1&0&0&0\\0&0&1&0\end{smallmatrix}\right)&
\left(\begin{smallmatrix}0&0&1&0\\0&1&1&1\\0&0&0&1\end{smallmatrix}\right)&
\left(\begin{smallmatrix}1&1&0&0\\0&0&1&0\\1&0&0&1\end{smallmatrix}\right)&
\left(\begin{smallmatrix}0&0&1&1\\0&1&0&1\\1&1&0&0\end{smallmatrix}\right)&
\left(\begin{smallmatrix}1&1&1&1\\0&1&0&0\\1&0&0&0\end{smallmatrix}\right)&
\left(\begin{smallmatrix}0&1&0&0\\0&1&1&0\\1&1&1&0\end{smallmatrix}\right),\\[+2mm]
\left(\begin{smallmatrix}1&1&1&0\\0&0&0&1\\0&1&1&0\end{smallmatrix}\right)&
\left(\begin{smallmatrix}0&0&0&1\\1&0&1&1\\0&1&0&1\end{smallmatrix}\right)&
\left(\begin{smallmatrix}1&0&0&0\\1&1&1&1\\1&0&1&0\end{smallmatrix}\right)&
\left(\begin{smallmatrix}1&0&1&0\\0&0&1&1\\0&1&1&1\end{smallmatrix}\right)&
\left(\begin{smallmatrix}1&1&0&1\\1&1&0&0\\0&0&1&1\end{smallmatrix}\right)&
\left(\begin{smallmatrix}0&1&1&0\\1&1&1&0\\1&1&0&1\end{smallmatrix}\right)&
\left(\begin{smallmatrix}0&1&1&1\\1&0&1&0\\1&0&1&1\end{smallmatrix}\right)&
\left(\begin{smallmatrix}1&0&0&1\\1&1&0&1\\1&1&1&1\end{smallmatrix}\right)\Big\}\\[+4mm]
%26
\Big\{\left(\begin{smallmatrix}0&0&0&0\\0&0&0&0\\0&0&0&0\end{smallmatrix}\right)&
\left(\begin{smallmatrix}0&1&0&0\\0&0&1&0\\0&0&0&1\end{smallmatrix}\right)&
\left(\begin{smallmatrix}1&1&0&1\\0&1&0&0\\0&0&1&0\end{smallmatrix}\right)&
\left(\begin{smallmatrix}1&0&0&0\\1&0&0&1\\1&1&0&0\end{smallmatrix}\right)&
\left(\begin{smallmatrix}1&0&1&0\\0&1&1&0\\0&1&0&0\end{smallmatrix}\right)&
\left(\begin{smallmatrix}1&0&0&1\\0&0&1&1\\1&0&0&0\end{smallmatrix}\right)&
\left(\begin{smallmatrix}0&0&0&1\\0&1&0&1\\1&0&1&0\end{smallmatrix}\right)&
\left(\begin{smallmatrix}0&0&1&0\\1&1&0&1\\0&1&0&1\end{smallmatrix}\right),\\[+2mm]
\left(\begin{smallmatrix}0&1&1&0\\0&0&0&1\\1&0&1&1\end{smallmatrix}\right)&
\left(\begin{smallmatrix}0&0&1&1\\1&0&0&0\\1&1&1&1\end{smallmatrix}\right)&
\left(\begin{smallmatrix}0&1&0&1\\1&0&1&1\\0&1&1&1\end{smallmatrix}\right)&
\left(\begin{smallmatrix}1&1&0&0\\1&1&1&1\\0&1&1&0\end{smallmatrix}\right)&
\left(\begin{smallmatrix}1&1&1&0\\1&1&0&0\\1&1&0&1\end{smallmatrix}\right)&
\left(\begin{smallmatrix}1&0&1&1\\1&1&1&0\\0&0&1&1\end{smallmatrix}\right)&
\left(\begin{smallmatrix}0&1&1&1\\1&0&1&0\\1&1&1&0\end{smallmatrix}\right)&
\left(\begin{smallmatrix}1&1&1&1\\0&1&1&1\\1&0&0&1\end{smallmatrix}\right)\Big\}\\[+4mm]
%27
\Big\{\left(\begin{smallmatrix}0&0&0&0\\0&0&0&0\\0&0&0&0\end{smallmatrix}\right)&
\left(\begin{smallmatrix}0&0&1&0\\0&1&0&0\\0&0&0&1\end{smallmatrix}\right)&
\left(\begin{smallmatrix}0&1&1&0\\1&0&0&0\\0&0&1&0\end{smallmatrix}\right)&
\left(\begin{smallmatrix}0&1&0&1\\1&1&0&0\\0&0&1&1\end{smallmatrix}\right)&
\left(\begin{smallmatrix}1&0&0&0\\1&1&1&1\\0&1&0&0\end{smallmatrix}\right)&
\left(\begin{smallmatrix}0&1&1&1\\1&0&1&0\\1&0&0&0\end{smallmatrix}\right)&
\left(\begin{smallmatrix}0&0&0&1\\1&0&1&1\\1&1&0&0\end{smallmatrix}\right)&
\left(\begin{smallmatrix}1&0&0&1\\0&0&0&1\\0&1&1&1\end{smallmatrix}\right),\\[+2mm]
\left(\begin{smallmatrix}1&1&0&1\\0&0&1&0\\0&1&1&0\end{smallmatrix}\right)&
\left(\begin{smallmatrix}1&0&1&1\\1&0&0&1\\1&0&1&0\end{smallmatrix}\right)&
\left(\begin{smallmatrix}1&1&1&0\\0&1&0&1\\1&0&0&1\end{smallmatrix}\right)&
\left(\begin{smallmatrix}0&1&0&0\\0&1&1&1\\1&1&1&0\end{smallmatrix}\right)&
\left(\begin{smallmatrix}1&1&0&0\\1&1&0&1\\0&1&0&1\end{smallmatrix}\right)&
\left(\begin{smallmatrix}0&0&1&1\\0&1&1&0\\1&0&1&1\end{smallmatrix}\right)&
\left(\begin{smallmatrix}1&0&1&0\\1&1&1&0\\1&1&1&1\end{smallmatrix}\right)&
\left(\begin{smallmatrix}1&1&1&1\\0&0&1&1\\1&1&0&1\end{smallmatrix}\right)\Big\}\\[+4mm]
%28
\Big\{\left(\begin{smallmatrix}0&0&0&0\\0&0&0&0\\0&0&0&0\end{smallmatrix}\right)&
\left(\begin{smallmatrix}0&0&1&1\\0&1&0&1\\1&0&0&0\end{smallmatrix}\right)&
\left(\begin{smallmatrix}0&1&0&0\\0&1&1&0\\1&0&1&0\end{smallmatrix}\right)&
\left(\begin{smallmatrix}0&1&1&0\\0&0&1&0\\0&0&1&1\end{smallmatrix}\right)&
\left(\begin{smallmatrix}1&0&0&1\\1&1&1&0\\0&1&0&0\end{smallmatrix}\right)&
\left(\begin{smallmatrix}1&0&1&0\\0&1&0&0\\1&0&1&1\end{smallmatrix}\right)&
\left(\begin{smallmatrix}0&0&0&1\\1&1&0&1\\0&1&1&0\end{smallmatrix}\right)&
\left(\begin{smallmatrix}1&0&1&1\\0&0&1&1\\0&0&0&1\end{smallmatrix}\right),\\[+2mm]
\left(\begin{smallmatrix}1&0&0&0\\1&0&1&1\\1&1&0&0\end{smallmatrix}\right)&
\left(\begin{smallmatrix}0&0&1&0\\1&0&0&1\\1&1&0&1\end{smallmatrix}\right)&
\left(\begin{smallmatrix}1&1&1&0\\0&1&1&1\\0&0&1&0\end{smallmatrix}\right)&
\left(\begin{smallmatrix}0&1&1&1\\1&0&0&0\\1&1&1&0\end{smallmatrix}\right)&
\left(\begin{smallmatrix}1&1&0&1\\1&1&0&0\\0&1&0&1\end{smallmatrix}\right)&
\left(\begin{smallmatrix}1&1&0&0\\0&0&0&1\\1&1&1&1\end{smallmatrix}\right)&
\left(\begin{smallmatrix}1&1&1&1\\1&0&1&0\\1&0&0&1\end{smallmatrix}\right)&
\left(\begin{smallmatrix}0&1&0&1\\1&1&1&1\\0&1&1&1\end{smallmatrix}\right)\Big\}\\[+4mm]
%29
\Big\{\left(\begin{smallmatrix}0&0&0&0\\0&0&0&0\\0&0&0&0\end{smallmatrix}\right)&
\left(\begin{smallmatrix}0&0&1&1\\0&0&0&1\\0&1&0&0\end{smallmatrix}\right)&
\left(\begin{smallmatrix}0&1&0&0\\0&0&1&0\\1&1&0&0\end{smallmatrix}\right)&
\left(\begin{smallmatrix}1&0&0&0\\0&0&1&1\\1&0&1&0\end{smallmatrix}\right)&
\left(\begin{smallmatrix}1&1&0&0\\0&1&0&0\\0&1&0&1\end{smallmatrix}\right)&
\left(\begin{smallmatrix}0&0&0&1\\1&1&0&1\\0&0&1&0\end{smallmatrix}\right)&
\left(\begin{smallmatrix}0&1&1&1\\1&0&0&1\\0&0&0&1\end{smallmatrix}\right)&
\left(\begin{smallmatrix}0&0&1&0\\0&1&1&0\\1&1&0&1\end{smallmatrix}\right),\\[+2mm]
\left(\begin{smallmatrix}0&1&1&0\\1&0&1&1\\1&0&0&0\end{smallmatrix}\right)&
\left(\begin{smallmatrix}1&0&1&0\\0&1&0&1\\0&1&1&1\end{smallmatrix}\right)&
\left(\begin{smallmatrix}1&1&0&1\\1&1&0&0\\0&1&1&0\end{smallmatrix}\right)&
\left(\begin{smallmatrix}1&0&0&1\\1&1&1&0\\0&0&1&1\end{smallmatrix}\right)&
\left(\begin{smallmatrix}0&1&0&1\\1&1&1&1\\1&0&0&1\end{smallmatrix}\right)&
\left(\begin{smallmatrix}1&0&1&1\\1&0&0&0\\1&1&1&1\end{smallmatrix}\right)&
\left(\begin{smallmatrix}1&1&1&1\\1&0&1&0\\1&1&1&0\end{smallmatrix}\right)&
\left(\begin{smallmatrix}1&1&1&0\\0&1&1&1\\1&0&1&1\end{smallmatrix}\right)\Big\}\\[+4mm]
%30
\Big\{\left(\begin{smallmatrix}0&0&0&0\\0&0&0&0\\0&0&0&0\end{smallmatrix}\right)&
\left(\begin{smallmatrix}1&0&1&0\\0&1&1&0\\0&1&0&0\end{smallmatrix}\right)&
\left(\begin{smallmatrix}0&1&1&0\\1&0&0&1\\0&0&1&0\end{smallmatrix}\right)&
\left(\begin{smallmatrix}1&1&0&0\\0&0&0&1\\1&0&0&1\end{smallmatrix}\right)&
\left(\begin{smallmatrix}1&0&0&0\\1&0&1&1\\0&0&0&1\end{smallmatrix}\right)&
\left(\begin{smallmatrix}0&0&0&1\\1&0&0&0\\0&1&1&1\end{smallmatrix}\right)&
\left(\begin{smallmatrix}0&1&0&0\\0&0&1&1\\1&1&1&0\end{smallmatrix}\right)&
\left(\begin{smallmatrix}0&1&1&1\\0&0&1&0\\0&1&1&0\end{smallmatrix}\right),\\[+2mm]
\left(\begin{smallmatrix}0&0&1&1\\0&1&0&1\\1&1&0&0\end{smallmatrix}\right)&
\left(\begin{smallmatrix}1&1&1&0\\0&1&0&0\\0&1&0&1\end{smallmatrix}\right)&
\left(\begin{smallmatrix}0&0&1&0\\1&0&1&0\\1&1&1&1\end{smallmatrix}\right)&
\left(\begin{smallmatrix}1&1&0&1\\1&1&0&0\\1&0&1&0\end{smallmatrix}\right)&
\left(\begin{smallmatrix}1&0&0&1\\0&1&1&1\\1&0&1&1\end{smallmatrix}\right)&
\left(\begin{smallmatrix}1&0&1&1\\1&1&1&1\\1&0&0&0\end{smallmatrix}\right)&
\left(\begin{smallmatrix}0&1&0&1\\1&1&1&0\\1&1&0&1\end{smallmatrix}\right)&
\left(\begin{smallmatrix}1&1&1&1\\1&1&0&1\\0&0&1&1\end{smallmatrix}\right)\Big\}\\[+4mm]
    \end{array}$
\end{table}

\section{Subspace codes}
\label{sect:codes}
We have $A_2(7,6) = 17$ and $A_2(7,5) = 34$ \cite{Honold-Kiermaier-Kurz-2016-AiMoC10[3]:649-682}.
Based on the classification results of Sections~\ref{sect:ps17} and~\ref{sect:ps16} we are now ready to classify the corresponding optimal $(7,17,6)_2$ and $(7,34,5)_2$ subspace codes.
This completes the work started in~\cite{Honold-Kiermaier-Kurz-2016-AiMoC10[3]:649-682}, where their numbers have already been announced.

Up to dualization the possible dimension distributions for a $(7,17,6)_2$ subspace code are $(3^{17})$, $(3^{16} 4^1)$, and $(3^{16} 5^1)$~\cite[Theorem~3.2(ii)]{Honold-Kiermaier-Kurz-2016-AiMoC10[3]:649-682}.
The only possible dimension distribution for a $(7,34,5)_2$ subspace code is $(3^{17} 4^{17})$~\cite[Remark~4]{Honold-Kiermaier-Kurz-2016-AiMoC10[3]:649-682}.

We are going to prove the following theorems:
\begin{theorem}
	\label{thm:7_17_6_2}
	There are $928$ isomorphism types of $(7,17,6)_2$ subspace codes.
	Of these codes, $715$ are constant dimension codes of dimension distribution $(3^{17})$, $37$ have dimension distribution $(3^{16} 4^1)$ and $176$ have dimension distribution $(3^{16} 5^1)$.
\end{theorem}

\begin{theorem}
	\label{thm:7_34_5_2}
	There are $20$ isomorphism types of $(7,34,5)_2$ subspace codes, all of the dimension distribution $(3^{17} 4^{17})$.
	Among the $20$ types of codes, there is a single iso-dual one.
	A representative of this code is given by the row spaces of the matrices in Table~\ref{tbl:7_34_5_2:selfdual}.
\end{theorem}

\begin{remark}
	According to the definition in Section~\ref{subsect:subspace_codes}, the acting group in Theorems~\ref{thm:7_17_6_2} and~\ref{thm:7_34_5_2} includes dualization.
	The number of inner isomorphism classes can easily be derived though.
	
	The codes in the context of Theorem~\ref{thm:7_17_6_2} are not iso-dual by the asymmetric dimension distribution.
	So the number of inner isomorphism classes is twice the stated number of isomorphism classes.
	For Theorem~\ref{thm:7_34_5_2}, the number of inner isomorphism classes of $(7,34,5)_2$ subspace codes is $1 + 2\cdot (20-1) = 39$, consisting of the single iso-dual type and $19$ dual pairs of types.
\end{remark}

\begin{table}
\caption{The unique iso-dual $(7,34,5)_2$ subspace code}
\label{tbl:7_34_5_2:selfdual}
\vspace{2mm}
\centering$\setlength{\arraycolsep}{2pt}
\begin{array}{ccccc}
\left(\begin{smallmatrix} 
0 & 0 & 0 & 0 & 1 & 0 & 0 \\
0 & 0 & 0 & 0 & 0 & 1 & 0 \\
0 & 0 & 0 & 0 & 0 & 0 & 1 
\end{smallmatrix}\right)\text{,} &

\left(\begin{smallmatrix} 
0 & 1 & 0 & 0 & 0 & 0 & 0 \\
0 & 0 & 1 & 0 & 0 & 0 & 0 \\
0 & 0 & 0 & 1 & 0 & 0 & 0 
\end{smallmatrix}\right)\text{,} &

\left(\begin{smallmatrix} 
0 & 1 & 0 & 0 & 1 & 0 & 0 \\
0 & 0 & 1 & 0 & 0 & 1 & 0 \\
0 & 0 & 0 & 1 & 0 & 0 & 1 
\end{smallmatrix}\right)\text{,} &

\left(\begin{smallmatrix} 
1 & 0 & 0 & 0 & 1 & 1 & 1 \\
0 & 0 & 1 & 0 & 0 & 0 & 1 \\
0 & 0 & 0 & 1 & 1 & 1 & 0 
\end{smallmatrix}\right)\text{,} &

\left(\begin{smallmatrix} 
1 & 0 & 1 & 0 & 1 & 0 & 1 \\
0 & 1 & 1 & 0 & 0 & 1 & 1 \\
0 & 0 & 0 & 1 & 0 & 1 & 0 
\end{smallmatrix}\right)\text{,} \\[1.5ex]

\left(\begin{smallmatrix} 
1 & 0 & 0 & 0 & 0 & 1 & 0 \\
0 & 1 & 0 & 0 & 0 & 1 & 0 \\
0 & 0 & 1 & 0 & 1 & 1 & 0 
\end{smallmatrix}\right)\text{,} &

\left(\begin{smallmatrix} 
1 & 0 & 0 & 0 & 1 & 0 & 1 \\
0 & 1 & 0 & 0 & 1 & 1 & 0 \\
0 & 0 & 1 & 0 & 1 & 0 & 0 
\end{smallmatrix}\right)\text{,} &

\left(\begin{smallmatrix} 
1 & 0 & 0 & 0 & 1 & 1 & 0 \\
0 & 1 & 0 & 0 & 1 & 1 & 1 \\
0 & 0 & 0 & 1 & 1 & 0 & 0 
\end{smallmatrix}\right)\text{,} &

\left(\begin{smallmatrix} 
1 & 0 & 0 & 0 & 0 & 1 & 1 \\
0 & 1 & 0 & 1 & 0 & 0 & 1 \\
0 & 0 & 1 & 1 & 1 & 1 & 0 
\end{smallmatrix}\right)\text{,} &

\left(\begin{smallmatrix} 
1 & 0 & 0 & 0 & 1 & 0 & 0 \\
0 & 1 & 0 & 1 & 1 & 0 & 0 \\
0 & 0 & 1 & 1 & 1 & 0 & 1 
\end{smallmatrix}\right)\text{,} \\[1.5ex]

\left(\begin{smallmatrix} 
1 & 0 & 0 & 1 & 0 & 1 & 1 \\
0 & 1 & 0 & 1 & 1 & 1 & 1 \\
0 & 0 & 1 & 0 & 1 & 0 & 1 
\end{smallmatrix}\right)\text{,} &

\left(\begin{smallmatrix} 
1 & 0 & 1 & 0 & 0 & 0 & 0 \\
0 & 1 & 1 & 0 & 1 & 0 & 1 \\
0 & 0 & 0 & 1 & 0 & 1 & 1 
\end{smallmatrix}\right)\text{,} &

\left(\begin{smallmatrix} 
1 & 0 & 0 & 1 & 1 & 1 & 0 \\
0 & 1 & 0 & 0 & 1 & 0 & 1 \\
0 & 0 & 1 & 1 & 0 & 0 & 1 
\end{smallmatrix}\right)\text{,} &

\left(\begin{smallmatrix} 
1 & 0 & 0 & 1 & 1 & 0 & 1 \\
0 & 1 & 0 & 1 & 0 & 1 & 0 \\
0 & 0 & 1 & 0 & 1 & 1 & 1 
\end{smallmatrix}\right)\text{,} &

\left(\begin{smallmatrix} 
1 & 0 & 0 & 0 & 0 & 0 & 0 \\
0 & 0 & 1 & 0 & 0 & 1 & 1 \\
0 & 0 & 0 & 1 & 1 & 1 & 1 
\end{smallmatrix}\right)\text{,} \\[1.5ex]

\left(\begin{smallmatrix} 
1 & 0 & 0 & 0 & 0 & 0 & 1 \\
0 & 1 & 0 & 0 & 0 & 1 & 1 \\
0 & 0 & 0 & 1 & 1 & 0 & 1 
\end{smallmatrix}\right)\text{,} &

\left(\begin{smallmatrix} 
1 & 0 & 0 & 1 & 0 & 0 & 0 \\
0 & 1 & 0 & 0 & 0 & 0 & 1 \\
0 & 0 & 1 & 1 & 0 & 1 & 0 \\
\end{smallmatrix}\right)\text{,} \\[2ex]

\left(\begin{smallmatrix} 
1 & 0 & 0 & 0 & 0 & 0 & 0 \\
0 & 1 & 0 & 0 & 0 & 1 & 0 \\
0 & 0 & 0 & 1 & 0 & 1 & 1 \\
0 & 0 & 0 & 0 & 1 & 0 & 1 
\end{smallmatrix}\right)\text{,} &

\left(\begin{smallmatrix} 
1 & 0 & 1 & 0 & 1 & 0 & 0 \\
0 & 1 & 1 & 0 & 0 & 0 & 0 \\
0 & 0 & 0 & 1 & 1 & 0 & 1 \\
0 & 0 & 0 & 0 & 0 & 1 & 0 
\end{smallmatrix}\right)\text{,} &

\left(\begin{smallmatrix} 
1 & 0 & 1 & 0 & 0 & 0 & 0 \\
0 & 1 & 0 & 0 & 0 & 1 & 1 \\
0 & 0 & 0 & 1 & 0 & 0 & 1 \\
0 & 0 & 0 & 0 & 1 & 0 & 1 
\end{smallmatrix}\right)\text{,} &

\left(\begin{smallmatrix} 
1 & 0 & 0 & 1 & 0 & 1 & 0 \\
0 & 1 & 0 & 1 & 0 & 1 & 0 \\
0 & 0 & 1 & 0 & 0 & 1 & 0 \\
0 & 0 & 0 & 0 & 1 & 1 & 0 
\end{smallmatrix}\right)\text{,} &

\left(\begin{smallmatrix} 
0 & 1 & 0 & 0 & 0 & 0 & 1 \\
0 & 0 & 1 & 0 & 1 & 0 & 1 \\
0 & 0 & 0 & 1 & 0 & 0 & 0 \\
0 & 0 & 0 & 0 & 0 & 1 & 0 
\end{smallmatrix}\right)\text{,} \\[2ex]

\left(\begin{smallmatrix} 
1 & 1 & 0 & 0 & 0 & 0 & 0 \\
0 & 0 & 1 & 0 & 0 & 0 & 0 \\
0 & 0 & 0 & 1 & 1 & 0 & 1 \\
0 & 0 & 0 & 0 & 0 & 1 & 1 
\end{smallmatrix}\right)\text{,} &

\left(\begin{smallmatrix} 
1 & 0 & 0 & 1 & 0 & 0 & 0 \\
0 & 1 & 0 & 0 & 1 & 0 & 0 \\
0 & 0 & 1 & 0 & 1 & 0 & 0 \\
0 & 0 & 0 & 0 & 0 & 0 & 1 
\end{smallmatrix}\right)\text{,} &

\left(\begin{smallmatrix} 
1 & 0 & 0 & 1 & 0 & 1 & 0 \\
0 & 1 & 0 & 1 & 0 & 0 & 0 \\
0 & 0 & 1 & 1 & 0 & 0 & 1 \\
0 & 0 & 0 & 0 & 1 & 1 & 1 
\end{smallmatrix}\right)\text{,} &

\left(\begin{smallmatrix} 
1 & 0 & 0 & 0 & 0 & 0 & 1 \\
0 & 0 & 1 & 0 & 0 & 0 & 0 \\
0 & 0 & 0 & 1 & 0 & 1 & 0 \\
0 & 0 & 0 & 0 & 1 & 1 & 0 
\end{smallmatrix}\right)\text{,} &

\left(\begin{smallmatrix} 
1 & 0 & 0 & 0 & 0 & 1 & 1 \\
0 & 1 & 0 & 0 & 0 & 0 & 0 \\
0 & 0 & 1 & 1 & 0 & 1 & 1 \\
0 & 0 & 0 & 0 & 1 & 1 & 1 
\end{smallmatrix}\right)\text{,} \\[2ex]

\left(\begin{smallmatrix} 
1 & 0 & 0 & 0 & 0 & 0 & 1 \\
0 & 1 & 0 & 0 & 0 & 0 & 0 \\
0 & 0 & 1 & 0 & 0 & 1 & 1 \\
0 & 0 & 0 & 0 & 1 & 0 & 0 
\end{smallmatrix}\right)\text{,} &

\left(\begin{smallmatrix} 
1 & 0 & 0 & 0 & 0 & 1 & 0 \\
0 & 1 & 0 & 1 & 0 & 0 & 1 \\
0 & 0 & 1 & 1 & 0 & 0 & 0 \\
0 & 0 & 0 & 0 & 1 & 0 & 0 
\end{smallmatrix}\right)\text{,} &

\left(\begin{smallmatrix} 
1 & 0 & 0 & 0 & 1 & 0 & 0 \\
0 & 1 & 0 & 0 & 1 & 0 & 1 \\
0 & 0 & 1 & 0 & 0 & 0 & 1 \\
0 & 0 & 0 & 1 & 0 & 1 & 1 
\end{smallmatrix}\right)\text{,} &

\left(\begin{smallmatrix} 
1 & 0 & 0 & 0 & 1 & 1 & 0 \\
0 & 1 & 0 & 0 & 0 & 0 & 1 \\
0 & 0 & 1 & 0 & 1 & 1 & 0 \\
0 & 0 & 0 & 1 & 1 & 1 & 1 
\end{smallmatrix}\right)\text{,} &

\left(\begin{smallmatrix} 
1 & 0 & 0 & 0 & 0 & 0 & 0 \\
0 & 1 & 0 & 1 & 0 & 0 & 0 \\
0 & 0 & 1 & 0 & 0 & 0 & 1 \\
0 & 0 & 0 & 0 & 0 & 1 & 1 
\end{smallmatrix}\right)\text{,} \\[2ex]

\left(\begin{smallmatrix} 
1 & 0 & 0 & 0 & 1 & 1 & 0 \\
0 & 1 & 1 & 0 & 0 & 1 & 0 \\
0 & 0 & 0 & 1 & 0 & 0 & 0 \\
0 & 0 & 0 & 0 & 0 & 0 & 1 
\end{smallmatrix}\right)\text{,} &

\left(\begin{smallmatrix} 
1 & 0 & 0 & 1 & 0 & 0 & 0 \\
0 & 1 & 0 & 0 & 0 & 1 & 0 \\
0 & 0 & 1 & 1 & 0 & 1 & 1 \\
0 & 0 & 0 & 0 & 1 & 0 & 0
\end{smallmatrix}\right)\phantom{,}
\end{array}
$
\end{table}

\subsection{$(7,17,6)_2$ subspace codes, dimension distribution $(3^{17})$}
These are exactly the maximum partial plane spreads in $\PG(6,2)$, whose number of isomorphism types has been determined as $715$ in Theorem~\ref{thm:size17}.

\subsection{$(7,17,6)_2$ subspace codes, dimension distribution $(3^{16} 4^1)$}
\label{sec:ssc:3_16__4_1}
The $16$ blocks of dimension $3$ form a partial plane spread $\mathcal{S}_{16}$ in $\PG(6,2)$.
The block of dimension $4$ has subspace distance at least $6$ to all the blocks in $\mathcal{S}_{16}$, implying that its intersection with each block in $\mathcal{S}_{16}$ is trivial.
In other words, the block of dimension $4$ is the set $N$ of holes of $\mathcal{S}_{16}$.
This means that the codes with the given dimension distribution are exactly the vector space partitions of $V$ of type $3^{16} 4^1$.

Thus, we have yet another characterization of the $37$ objects counted in Theorem~\ref{thm:size16}\ref{thm:size16:ext:n4} and Theorem~\ref{thm:mrd} as the types of $(7,17,6)_2$ subspace codes of dimension distribution $(3^{16} 4^1)$.

\subsection{$(7,17,6)_2$ subspace codes, dimension distribution $(3^{16} 5^1)$}
\begin{lemma}
	The $(7,17,6)_2$ subspace codes of dimension distribution $(3^{16} 5^1)$ are exactly the sets of the form $\mathcal{S}_{16} \cup \{F\}$, where $\mathcal{S}_{16}$ is a partial plane spread of size $16$ and $F$ is a $5$-space containing all the holes of $\mathcal{S}_{16}$.
\end{lemma}

\begin{proof}
Let $C$ be a $(7,17,6)_2$ subspace code.
Then, by the subspace distance $6$, we have $C = \mathcal{S}_{16} \cup \{F\}$ with a partial plane spread $\mathcal{S}_{16}$ and a $5$-subspace $F$ such that for any block $B\in\mathcal{S}_{16}$, the intersection $B\cap F$ is at most a point.
So there remain at least $31 - 16 = 15$ points of $F$, which are holes of $\mathcal{S}_{16}$.
Since the total number of holes is only $15$, all the holes have to be contained in $F$.
\end{proof}

By the above lemma, we can enumerate all subspace codes of dimension distribution $(3^{16} 5^1)$ by extending all partial plane spreads $\mathcal{S}_{16}$ of hole space dimension $n = \dim\langle N\rangle \leq 5$ by a $5$-flat $F \supseteq N$.
Based on the classification of the partial plane spreads $\mathcal{S}_{16}$ in Theorem~\ref{thm:size16}, there are the following two cases:
\begin{enumerate}[(i)]
	\item $n = 5$. \\
	Here $F = \langle N\rangle$ is uniquely determined, so the $69$ types of partial plane spreads $\mathcal{S}_{16}$ with $n=5$ yield $69$ types of subspace codes.
	\item $n = 4$. \\
	Here, $F$ may be taken as any of the seven $5$-spaces in $\PG(V)$ passing through $\langle N\rangle$.
	For each of the $37$ partial plane spreads $\mathcal{S}_{16}$ with $n = 4$, we check computationally if there arise equivalences among the $7$ produced subspace codes.
	The resulting pattern is
	\[
		(7^1)^3
		(6^1 1^1)^1
		(4^1 3^1)^6
		(4^1 2^1 1^1)^9
		(4^1 1^3)^1
		(3^2 1^1)^{12}
		(2^3 1^1)^2
		(2^2 1^3)^3
	\]
	For example, the part $(3^2 1^1)^{12}$ means that among the $37$ partial plane spreads there are $12$, where out of the $7$ possibilities for $F$, $3$ lead to the same isomorphism type of a $(7,17,6)_2$ subspace code, $3$ others lead to a second isomorphism type and a single possibility for $F$ leads to a third isomorphism type.
	In other words, for $12$ of the $37$ partial plane spreads, the automorphism group partitions the set of seven $5$-flats through $N$ into $2$ orbits of length $3$ and $1$ orbit of length $1$.
	Together, we arrive at $107$ types of subspace codes.
\end{enumerate}

In this way, we get:
\begin{theorem}
	\label{thm:code_3_5}
	There are $176$ isomorphism types of $(7,17,6)_2$ subspace
        codes of dimension distribution $(3^{16} 5^1)$.  The $16$
        codewords of dimension $3$ form a partial plane spread in
        $\PG(6,2)$.  In $107$ cases the hole space dimension is $4$,
        and in $69$ cases the hole space dimension is $5$.
\end{theorem}

\subsection{$(7,34,5)_2$ subspace codes}
As stated above, the dimension distribution is $(3^{17} 4^{17})$.
Thus, any $(7,34,5)_2$ subspace code is the union of a maximal partial
plane spread and the dual of a maximal partial plane spread in
$\PG(6,2)$.

For their construction, we computationally checked the $715$ types of
maximal partial plane spreads from Theorem~\ref{thm:size17} for the
extendibility to a $(7,34,5)_2$ subspace code by a maximal clique
search.  It turned out that this is possible only in $9$ cases, with
the number of extensions being $16$, $768$, $192$, $192$, $2824$,
$12$, $64$, $13$, and $6$.  Filtering out $\PGammaL$-isomorphic
copies, the number of extensions is $1$, $2$, $6$, $6$, $2$, $2$, $4$,
$13$ and $4$, comprising $39$ types in total.  Adding dualization to
the group of isomorphisms, we remain with $20$ codes.  This implies
that there is a single iso-dual code.

\section*{Acknowledgement}
The authors would like to acknowledge the financial support provided by COST -- \emph{European Cooperation in Science and Technology}.
The first author was also supported by the National Natural Science Foundation of China under Grant 61571006.
The second and the third author are members of the Action IC1104 \emph{Random Network Coding and Designs over GF(q)}.
The third author was supported in part by the grant KU 2430/3-1 -- Integer Linear Programming Models for Subspace Codes and Finite Geometry from the German Research Foundation.

%%\small

%%\bibliography{netcod16_full}

\begin{thebibliography}{10}

\bibitem{Beutelspacher-1975-MathZeit145[3]:211-229}
Albrecht Beutelspacher.
\newblock Partial spreads in finite projective spaces and partial designs.
\newblock {\em Math. Z.}, 145(3):211--229, 1975.

\bibitem{Magma}
Wieb Bosma, John Cannon, and Catherine Playoust.
\newblock The {M}agma algebra system. {I}. the user language.
\newblock {\em J. Symbolic Comput.}, 24(3--4):235--265, 1997.

\bibitem{Braun-Kiermaier-Wassermann-2018-COST2}
Michael Braun, Michael Kiermaier, and Alfred Wassermann.
\newblock Computational methods in subspace designs.
\newblock In Marcus Greferath, Mario~Osvin Pavčević, Natalia Silberstein, and
  María Ángeles Vázquez-Castro, editors, {\em Network Coding and Subspace
  Designs}, Signals and Communication Theory, pages 213--244. Springer, Cham,
  2018.

\bibitem{Braun-Kiermaier-Wassermann-2018-COST1}
Michael Braun, Michael Kiermaier, and Alfred Wassermann.
\newblock $q$-analogs of designs: Subspace designs.
\newblock In Marcus Greferath, Mario~Osvin Pavčević, Natalia Silberstein, and
  María Ángeles Vázquez-Castro, editors, {\em Network Coding and Subspace
  Designs}, Signals and Communication Theory, pages 171--211. Springer, Cham,
  2018.

\bibitem{DeLaCruz-Kiermaier-Wassermann-Willems-2016-AiMoC10[3]:499-510}
Javier de~la Cruz, Michael Kiermaier, Alfred Wassermann, and Wolfgang Willems.
\newblock Algebraic structures of {M}{R}{D} codes.
\newblock {\em Adv. Math. Commun.}, 10(3):499--510, 2016.

\bibitem{Delsarte-1978-JCTSA25[3]:226-241}
Philippe Delsarte.
\newblock Bilinear forms over a finite field, with applications to coding
  theory.
\newblock {\em J. Combin. Theory Ser. A}, 25(3):226--241, 1978.

\bibitem{Dodunekov-Simonis-1998-ElecJComb5:R37}
Stefan Dodunekov and Juriaan Simonis.
\newblock Codes and projective multisets.
\newblock {\em Electron. J. Combin.}, 5:\#R37, 1998.

\bibitem{Eisfeld-Storme-2000}
J\"org Eisfeld and Leo Storme.
\newblock ({P}artial) $t$-spreads and minimal $t$-covers in finite projective
  spaces.
\newblock Lecture notes, Ghent University, 2000.

\bibitem{ElZanati-Heden-Seelinger-Sissokho-Spence-VandenEynden-2010-JCD18[6]:462
-474}
Saad El-Zanati, Olof Heden, George Seelinger, Papa Sissokho, Lawrence Spence,
  and Charles Vanden~Eynden.
\newblock Partitions of the $8$-dimensional vector space over
  {$\operatorname{GF}(2)$}.
\newblock {\em J. Combin. Des.}, 18(6):462--474, 2010.

\bibitem{ElZanati-Jordon-Seelinger-Sissokho-Spence-2010-DCC54[2]:101-107}
Saad El-Zanati, Heather Jordon, George Seelinger, Papa Sissokho, and Lawrence
  Spence.
\newblock The maximum size of a partial $3$-spread in a finite vector space
  over {$\operatorname{GF}(2)$}.
\newblock {\em Des. Codes Cryptogr.}, 54(2):101--107, 2010.

\bibitem{Etzion-Silberstein-2013-IEEETIT59[2]:1004-1017}
Tuvi Etzion and Natalia Silberstein.
\newblock Codes and designs related to lifted {M}{R}{D} codes.
\newblock {\em {IEEE} Trans. Inf. Theory}, 59(2):1004--1017, 2013.

\bibitem{Etzion-Storme-2016-DCC78[1]:311-350}
Tuvi Etzion and Leo Storme.
\newblock Galois geometries and coding theory.
\newblock {\em Des. Codes Cryptogr.}, 78(1):311--350, 2016.

\bibitem{Etzion-Vardy-2011-IEEETIT57[2]:1165-1173}
Tuvi Etzion and Alexander Vardy.
\newblock Error-correcting codes in projective spaces.
\newblock {\em {IEEE} Trans. Inf. Theory}, 57(2):1165--1173, 2011.

\bibitem{Feulner-2009-AiMoC3[4]:363-383}
Thomas Feulner.
\newblock The automorphism groups of linear codes and canonical representatives
  of their semilinear isometry classes.
\newblock {\em Adv. Math. Commun.}, 3(4):363--383, 2009.

\bibitem{Feulner-2013-arXiv:1305.1193}
Thomas Feulner.
\newblock Canonical forms and automorphisms in the projective space.
\newblock arXiv:1305.1193, 2013.

\bibitem{Feulner-2013-Thesis}
Thomas Feulner.
\newblock {\em Eine kanonische Form zur {D}arstellung {\"a}quivalenter {C}odes
  -- {C}omputergest{\"u}tzte {B}erechnung und ihre {A}nwendung in der
  {C}odierungstheorie, {K}ryptographie und {G}eometrie}.
\newblock PhD thesis, Universit{\"{a}}t Bayreuth, 2013.

\bibitem{Gabidulin-1985-PInfTr21[1]:1-12}
Ernst~M. Gabidulin.
\newblock Theory of codes with maximum rank distance.
\newblock {\em Problems Inform. Transmission}, 21(1):1--12, 1985.

\bibitem{Gordon-Shaw-Soicher-2004-unpublished}
Neil~A. Gordon, Ron Shaw, and Leonard~H. Soicher.
\newblock Classification of partial spreads in {$\operatorname{PG}(4,2)$}.
\newblock \url{www.maths.qmul.ac.uk/~leonard/partialspreads/PG42new.pdf}, 2004.

\bibitem{Heden-2009-DM309[21]:6169-6180}
Olof Heden.
\newblock On the length of the tail of a vector space partition.
\newblock {\em Discrete Math.}, 309(21):6169--6180, 2009.

\bibitem{Heden-2012-DiscreteMathAlgorithmsAppl4[1]:1250001}
Olof Heden.
\newblock A survey of the different types of vector space partitions.
\newblock {\em Discrete Math. Algorithms Appl.}, 4(1):1250001, 2012.

\bibitem{Heinlein-2018-arXiv:1801.04803}
Daniel Heinlein.
\newblock New {L}{M}{R}{D} bounds for constant dimension codes and improved
  constructions.
\newblock arXiv:1801.04803, 2018.

\bibitem{Heinlein-Honold-Kiermaier-Kurz-Wassermann-2017-arXiv:1711.06624}
Daniel Heinlein, Thomas Honold, Michael Kiermaier, Sascha Kurz, and Alfred
  Wassermann.
\newblock Classifying optimal binary subspace codes of length 8, constant
  dimension 4 and minimum distance 6.
\newblock arXiv:1711.06624, 2017.

\bibitem{Heinlein-Kiermaier-Kurz-Wassermann-2016-arXiv:1601.02864}
Daniel Heinlein, Michael Kiermaier, Sascha Kurz, and Alfred Wassermann.
\newblock Tables of subspace codes.
\newblock arXiv:1601.02864, 2016.

\bibitem{Heinlein-Kurz-2017-IEEETIT63[12]:7651-7660}
Daniel Heinlein and Sascha Kurz.
\newblock Coset construction for subspace codes.
\newblock {\em {IEEE} Trans. Inf. Theory}, 63(12):7651--7660, 2017.

\bibitem{Heinlein-Kurz-2017-WCC}
Daniel Heinlein and Sascha Kurz.
\newblock An upper bound for binary subspace codes of length $8$, constant
  dimension $4$ and minimum distance $6$.
\newblock Accepted paper at The Tenth International Workshop on Coding and
  Cryptography, September 18--22 2017, Saint-Petersburg, Russia., 2017.

\bibitem{Hirschfeld-1998}
James W.~P. Hirschfeld.
\newblock {\em Projective geometries over finite fields}.
\newblock Oxford Mathematical Monographs. The Clarendon Press, Oxford
  University Press, New York, second edition, 1998.

\bibitem{Hitotumatu-Noshita-1979-IPL8[4]:174-175}
Hirosi Hitotumatu and Kohei Noshita.
\newblock A technique for implementing backtrack algorithms and its
  application.
\newblock {\em Inform. Process. Lett.}, 8(4):174--175, 1979.

\bibitem{Hong-Patel-1972-IEEEToC_C21[12]:1322-1331}
Se~June Hong and Arvind~M. Patel.
\newblock A general class of maximal codes for computer applications.
\newblock {\em {IEEE} Trans. Comput.}, C-21(12):1322--1331, 1972.

\bibitem{Honold-Kiermaier-Kurz-2015-ContempM632:157-176}
Thomas Honold, Michael Kiermaier, and Sascha Kurz.
\newblock Optimal binary subspace codes of length $6$, constant dimension $3$
  and minimum subspace distance $4$.
\newblock In Gohar Kyureghyan, Gary~L. Mullen, and Alexander Pott, editors,
  {\em Topics in Finite Fields}, number 632 in Contemporary Mathematics, pages
  157--176. American Mathematical Society, 2015.

\bibitem{Honold-Kiermaier-Kurz-2016-AiMoC10[3]:649-682}
Thomas Honold, Michael Kiermaier, and Sascha Kurz.
\newblock Constructions and bounds for mixed-dimension subspace codes.
\newblock {\em Adv. Math. Commun.}, 10(3):649--682, 2016.

\bibitem{Honold-Kiermaier-Kurz-2018-COST}
Thomas Honold, Michael Kiermaier, and Sascha Kurz.
\newblock Partial spreads and vector space partitions.
\newblock In Marcus Greferath, Mario~Osvin Pavčević, Natalia Silberstein, and
  María Ángeles Vázquez-Castro, editors, {\em Network Coding and Subspace
  Designs}, Signals and Communication Theory, pages 131--170. Springer, Cham,
  2018.

\bibitem{Honold-Landjev-2000-ElecJComb7:R11}
Thomas Honold and Ivan Landjev.
\newblock Linear codes over finite chain rings.
\newblock {\em Electron. J. Combin.}, 7:\#R11, 2000.

\bibitem{HorlemannTrautmann-Marshall-2017-AiMoC11[3]:533-548}
Anna-Lena Horlemann-Trautmann and Kyle Marshall.
\newblock New criteria for {M}{R}{D} and {G}abidulin codes and some rank-metric
  code constructions.
\newblock {\em Adv. Math. Commun.}, 11(3):533--548, 2017.

\bibitem{Hua-1951-ActaMSinica1:109-163}
Loo-Keng Hua.
\newblock A theorem on matrices over a sfield and its applications.
\newblock {\em Acta Math. Sinica}, 1:109--163, 1951.

\bibitem{Kaski-Pottonen-2008-libexact}
Petteri Kaski and Olli Pottonen.
\newblock libexact user's guide version 1.0.
\newblock technical report 2008-1, Helsinki University of Technology, 2008.

\bibitem{Kiermaier-Laue-2015-AiMoC9[1]:105-115}
Michael Kiermaier and Reinhard Laue.
\newblock Derived and residual subspace designs.
\newblock {\em Adv. Math. Commun.}, 9(1):105--115, 2015.

\bibitem{Knuth-2000-dancing_links}
Donald~E. Knuth.
\newblock Dancing links.
\newblock In A.~W. Roscoe, J.~Davies, and J.~Woodcock, editors, {\em Millennial
  perspectives in computer science}, Cornerstones of computing, pages 187--214.
  Palgrave, 2000.

\bibitem{Koetter-Kschischang-2008-IEEETIT54[8]:3579-3591}
Ralf K{\"{o}}tter and Frank~R. Kschischang.
\newblock Coding for errors and erasures in random network coding.
\newblock {\em {IEEE} Trans. Inf. Theory}, 54(8):3579--3591, 2008.

\bibitem{Kshevetskiy-Gabidulin-2005-PIISIT2000:446}
Alexander Kshevetskiy and Ernst Gabidulin.
\newblock The new construction of rank codes.
\newblock In {\em Proceedings of the {IEEE} international symposium on
  information theory ({ISIT}) 2005}, pages 2105--2108, 2005.

\bibitem{Kurz-2017-DCC85[1]:97-106}
Sascha Kurz.
\newblock Improved upper bounds for partial spreads.
\newblock {\em Des. Codes Cryptogr.}, 85(1):97--106, 2017.

\bibitem{Kurz-2017-AustAsJComb68[1]:122-130}
Sascha Kurz.
\newblock Packing vector spaces into vector spaces.
\newblock {\em Australas. J. Combin.}, 68(1):122--130, 2017.

\bibitem{Liebhold-Nebe-2016-ArchMath107[4]:355-366}
Dirk Liebhold and Gabriele Nebe.
\newblock Automorphism groups of {G}abidulin-like codes.
\newblock {\em Arch. Math.}, 107(4):355--366, 2016.

\bibitem{Mateva-Topalova-2009-JCD17[1]:90-102}
Zlatka~T. Mateva and Svetlana~T. Topalova.
\newblock Line spreads of {$\operatorname{PG}(5,2)$}.
\newblock {\em J. Combin. Des.}, 17(1):90--102, 2009.

\bibitem{Nastase-Sissokho-2017-DM340[7]:1481-1487}
Esmeralda N{\u a}stase and Papa Sissokho.
\newblock The maximum size of a partial spread in a finite projective space.
\newblock {\em J. Combin. Theory Ser. A}, 152:353--362, 2017.

\bibitem{Niskanen-Ostergard-2003-cliquer}
Sampo Niskanen and Patric R.~J. {\"{O}}sterg\r{a}rd.
\newblock Cliquer user's guide, version 1.0.
\newblock technical report T48, Helsinki University of Technology, 2003.

\bibitem{Roth-1991-IEEETIT37[2]:328-336}
Ron~M. Roth.
\newblock Maximum-rank array codes and their application to crisscross error
  correction.
\newblock {\em {IEEE} Trans. Inf. Theory}, 37(2):328--336, 1991.

\bibitem{Seelinger-Sissokho-Spence-VandenEynden-2012-FFA18[6]:1114-1132}
George Seelinger, Papa Sissokho, Lawrence Spence, and Charles Vanden~Eynden.
\newblock Partitions of finite vector spaces over {$\operatorname{GF}(2)$} into
  subspaces of dimensions $2$ and $s$.
\newblock {\em Finite Fields Appl.}, 18(6):1114--1132, 2012.

\bibitem{Segre-1964-AnnDMPA64[1]:1-76}
Beniamino Segre.
\newblock Teoria di {G}alois, fibrazioni proiettive e geometrie non
  desarguesiane.
\newblock {\em Ann. Mat. Pura Appl. (4)}, 64(1):1--76, 1964.

\bibitem{Shaw-2000-DCC21[1-3]:209-222}
Ron Shaw.
\newblock Subsets of {$\operatorname{PG}(n,2)$} and maximal partial spreads in
  {$\operatorname{PG}(4,2)$}.
\newblock {\em Des. Codes Cryptogr.}, 21(1--3):209--222, 2000.

\bibitem{Sheekey-2016-AiMoC10[3]:475-488}
John Sheekey.
\newblock A new family of linear maximum rank distance codes.
\newblock {\em Adv. Math. Commun.}, 10(3):475--488, 2016.

\bibitem{Silva-Kschischang-Koetter-2008-IEEETIT54[9]:3951-3967}
Danilo Silva, Frank~R. Kschischang, and Ralf K{\"{o}}tter.
\newblock A rank-metric approach to error control in random network coding.
\newblock {\em {IEEE} Trans. Inf. Theory}, 54(9):3951--3967, 2008.

\bibitem{Wan-1962-SciSinica11:1183-1194}
Zhe-Xian Wan.
\newblock A proof of the automorphisms of linear groups over a sfield of
  characteristic $2$.
\newblock {\em Sci. Sinica}, 11:1183--1194, 1962.

\bibitem{Wan-1996-GeometryOfMatrices}
Zhe-Xian Wan.
\newblock {\em Geometry of matrices}.
\newblock World Scientific, 1996.

\end{thebibliography}

\end{document}